%% file: Bigprincipalseries.tex
\documentclass[a4paper]{amsart}
\usepackage[leqno]{amsmath}
\usepackage{amssymb}
\usepackage{amscd}
\usepackage{amsthm}
\usepackage{mathrsfs}
\usepackage{mathtools}
\usepackage{bbm}
\usepackage{color}
\usepackage{enumerate}
\usepackage{cite}
\usepackage[utf8]{inputenc}
\usepackage{stmaryrd}
\usepackage[all,cmtip]{xy}
\usepackage{etoolbox}
\usepackage{tikz}
\usepackage{hyperref}
\usepackage{url}
\numberwithin{equation}{section}
						
\newcommand{\Z}{\ensuremath{\mathbb{Z}}}
\newcommand{\Q}{\ensuremath{\mathbb{Q}}}

\newcommand{\C}{\ensuremath{\mathbb{C}}}

\newcommand{\A}{\ensuremath{\mathbb{A}}}
\newcommand{\T}{\ensuremath{\mathbb{T}}}

\DeclareMathOperator{\PGL}{PGL}
\DeclareMathOperator{\GL}{GL}
\DeclareMathOperator{\Sp}{Sp}

\DeclareMathOperator{\supp}{supp}
\DeclareMathOperator{\Hom}{Hom}
\DeclareMathOperator{\End}{End}
\DeclareMathOperator{\Tor}{Tor}
\DeclareMathOperator{\Ext}{Ext}

\DeclareMathOperator{\Gal}{Gal}
\DeclareMathOperator{\Res}{Res}

\DeclareMathOperator{\Ker}{ker}
\DeclareMathOperator{\Coker}{coker}
\DeclareMathOperator{\im}{im}
\DeclareMathOperator{\rk}{rk}

\DeclareMathOperator{\ord}{ord}
\DeclareMathOperator{\aug}{aug}
\DeclareMathOperator{\diag}{diag}

\DeclareMathOperator{\Spec}{Spec}
\DeclareMathOperator{\Spm}{Spm}
\DeclareMathOperator{\Std}{Std}

\DeclareMathOperator{\St}{St}

\DeclareMathOperator{\cind}{c-ind}
\DeclareMathOperator{\Coind}{Coind}

\DeclareMathOperator{\JL}{JL}
\DeclareMathOperator{\FM}{FM}
\DeclareMathOperator{\BC}{BC}

\DeclareMathOperator{\HH}{H}

\newcommand{\cf}{{\mathbbm 1}}

\newcommand{\II}{\ensuremath{\mathbb{I}}}	
\newcommand{\VV}{\ensuremath{\mathbb{V}}}

\newcommand{\pinfty}{\ensuremath{^{\p,\infty}}}

\newcommand{\Co}{E}
\newcommand{\Fu}{\ensuremath{\mathcal{C}}}
\newcommand{\An}{\ensuremath{\mathcal{A}}}
\newcommand{\Di}{\ensuremath{\mathcal{D}}}
\newcommand{\Ws}{\ensuremath{\mathcal{W}}}

\DeclareMathOperator{\red}{red}

\DeclareMathOperator{\add}{add}

\DeclareMathOperator{\cont}{ct}
\DeclareMathOperator{\an}{an}

\DeclareMathOperator{\sph}{sph}

\newcommand{\m}{\ensuremath{\mathfrak{m}}}
\newcommand{\p}{\ensuremath{\mathfrak{p}}}
\newcommand{\q}{\ensuremath{\mathfrak{q}}}

\newcommand{\LI}{\mathcal{L}}

\DeclareMathOperator{\univ}{un}

\DeclareMathOperator{\ev}{ev}

\newcommand{\into}{\hookrightarrow}
\newcommand{\onto}{\twoheadrightarrow}
\newcommand{\too}{\longrightarrow}								
\newcommand{\mapstoo}{\longmapsto}

\newtheorem{Lem}{Lemma}[section]
\makeatletter
\newlength{\@thlabel@width}%
\newcommand{\thmenumhspace}{\settowidth{\@thlabel@width}{\itshape1.}\sbox{\@labels}{\unhbox\@labels\hspace{\dimexpr-\leftmargin+\labelsep+\@thlabel@width-\itemindent}}}
\makeatother
\newtheorem{Pro}[Lem]{Proposition}

\newtheorem{Thm}[Lem]{Theorem}

\newtheorem{Def}[Lem]{Definition}
\newtheorem{Rem}[Lem]{Remark}

\newtheorem{Cor}[Lem]{Corollary}
\newtheorem{Exa}[Lem]{Example}
\newtheorem*{Hyp}{Hypothesis}

\author[L. Gehrmann]{Lennart Gehrmann}
\address{L. Gehrmann \\ Fakult\"at f\"ur Mathematik \\ Universit\"at Duisburg-Essen \\ Thea-Leymann-Stra\ss e 9 \\ 45127 Essen \\ Germany}
\email{lennart.gehrmann@uni-due.de}
\author[G. Rosso]{Giovanni Rosso}
\address{G. Rosso \\ Concordia University \\ Departments of Mathematics and Statistics \\ Montreal \\ Quebec \\ Canada}
\email{giovanni.rosso@concordia.ca}

\title[Big principal series, p-adic families and L-invariants]{Big principal series, p-adic families and L-invariants}

\begin{document}

\begin{abstract}
In earlier work, the first named author generalized the construction of Darmon-style $\LI$-invariants to cuspidal automorphic representations of semisimple groups of higher rank, which are cohomological with respect to the trivial coefficient system and Steinberg at a fixed prime. 

In this paper, assuming that the Archimedean component of the group has discrete series we show that these automorphic $\LI$-invariants can be computed in terms of derivatives of Hecke-eigenvalues in $p$-adic families.
Our proof is novel even in the case of modular forms, which was established by Bertolini, Darmon, and Iovita.
The main new technical ingredient is the Koszul resolution of locally analytic principal series representations by Kohlhaase and Schraen. 

As an application of our results we settle a conjecture of Spie\ss: we show that automorphic $\LI$-invariants of Hilbert modular forms of parallel weight $2$ are independent of the sign character used to define them.
Moreover, we show that they are invariant under Jacquet--Langlands transfer and, in fact, equal to the Fontaine--Mazur $\LI$-invariant of the associated Galois representation.
Under mild assumptions, we also prove the equality of automorphic and Fontaine--Mazur $\LI$-invariants for representations of definite unitary groups of arbitrary rank.

Finally, we study the case of Bianchi modular forms to show how our methods, given precise results on eigenvarieties, can also work in the absence of discrete series representations.
\end{abstract}

\maketitle

\tableofcontents

\input{Big-Introduction.tex}
\input{Big-Setup.tex}

\section{Automorphic L-invariants}\label{firstsection}
In the following we briefly sketch the construction of automorphic $\LI$-invariants. For more details see \cite{Ge4}.
\input{Big-AL-Extensions.tex}
\input{Big-AL-Resolutions.tex}

\input{Big-AL-Cohomology.tex}
\input{Big-AL-Definition.tex}

\section{L-invariants and big principal series}
By the Colmez--Greenberg--Stevens formula (see \cite{Ding}, Theorem 3.4) one can calculate Fontaine--Mazur $\LI$-invariants of certain local Galois representations by deforming them (or rather the attached $(\varphi,\Gamma)$-modules) in rigid analytic families.
The goal of this section is to prove an automorphic version of that formula.

This implies that $\LI$-invariants can be related to the existence of cohomology classes with values in (duals of) families of principal series.
\input{Big-Ref-Dualnumbers.tex}

\input{Big-Ref-Principalseries.tex}
\input{Big-Ref-Bigprincipalseries.tex}

\section{Overconvergent cohomology}
After giving a brief overview on Kohlhaase and Schraen's Koszul resolution of locally analytic principal series (see \cite{KoSch}) we recall the control theorem relating overconvergent cohomology to classical cohomology as proven by Ash--Stevens, Urban and Hansen (see \cite{AshSt}, \cite{Ur} and \cite{Hansen}).
Combining the two results allows us to construct classes in the cohomology of (duals of) principal series.
If $G_\infty$ fulfils the Harish-Chandra condition, i.e.~if $\delta=0$, we can lift the construction to families of principal series.
This implies our main theorem.
\input{Big-OC-Koszul.tex}

\input{Big-OC-Overconvergent.tex}

\input{Big-OC-Families.tex}

\section{Applications}
\input{Big-App-JL.tex}

\input{Big-App-Galois.tex}

\section{Beyond discrete series: the Bianchi case}
\input{Big-Bianchi.tex}

\bibliographystyle{alpha}
\bibliography{bibfile}

\end{document}

%% file: Big-Introduction.tex
\section*{Introduction}
Let $f=\sum_{n=1}^{\infty} a_n q^{n}$ be a normalized newform of weight $2$ and level $\Gamma_0(M)$ such that $M=pN$ with $p$ prime, $p\nmid N$ and $a_p=1$.
Inspired by Teitelbaum's work (cf.~\cite{Teitelbaum}) on $\LI$-invariants for automorphic forms on definite quaternion groups, Darmon in \cite{D} constructed the automorphic $\LI$-invariant $\LI(f)^{\pm}$ of $f$, depending a priori on a choice of sign at infinity.
Let $a_p(k)$ be the $U_p$-eigenvalue of the $p$-adic family passing through $f$.
In \cite{BDI} Bertolini, Darmon, and Iovita (see also \cite{Das05}) prove the following formula:
\begin{align}\label{basicformula}
\LI^{\pm}(f)=-2\left.\frac{d a_p}{dk}\right|_{k=2}.
\end{align}
In particular, the automorphic $\LI$-invariant is independent of the sign at infinity.
They also prove a similar formula for Orton's $\LI$-invariant (cf.~\cite{Orton}) in higher weight.  
Moreover, if $f$ admits a Jacquet-Langlands transfer $\JL(f)$ to a definite quaternion group, which is split at $p$, they show the analogous formula
$$\LI_T(\JL(f))=-2\left.\frac{d a_p}{dk}\right|_{k=2}$$
for Teitelbaum's $\LI$-invariant, thus proving that automorphic $\LI$-invariants are preserved under Jacquet--Langlands transfers.
This result was extended to Jacquet--Langlands transfers to indefinite quaternion groups over the rationals, which are split at $p$, by Dasgupta--Greenberg, Longo--Rotger--Vigni and Seveso (see \cite{DG}, \cite{LRV}, \cite{Seveso}).\\

Over the last couple of years the construction of automorphic $\LI$-invariants was generalized to various settings, e.g.~to Hilbert modular cusp forms of parallel weight $2$ by Spie\ss~(see \cite{Sp}) and to Bianchi modular cusp forms of even weight by Barrera and Williams (see \cite{BWi}).
These generalizations have been defined as automorphic $\LI$-invariants appear naturally in many context, from proving exceptional zero conjecture formulae to constructing and studying Stark--Heegner points. 

Most recently, the first named author defined automorphic $\LI$-invariants for certain cuspidal automorphic representations of higher rank semi-simple groups over a number field $F$, which are split at a fixed prime $\p$ of $F$ (see \cite{Ge4}).
The most crucial assumptions on the representation $\pi$ are that it is cohomological with respect to the trivial coefficient system and that the local factor $\pi_\p$ of $\pi$ at $\p$ is the Steinberg representation.
Our main aim is to prove the analogue of equation \eqref{basicformula} for these $\LI$-invariants.
Previous works used explicit computations with cocycles and it seems unlikely that one can generalize these to higher rank groups;
instead we give a new, more conceptual approach, which is novel even in the already known cases. \\

In Section \ref{firstsection} we recall the definition of automorphic $\LI$-invariants.
Just as in the case of modular forms these $\LI$-invariants depend on a choice of sign character at infinity.
They also depend on a choice of degree of cohomology, in which the representation occurs.
For this introduction we suppress it since we are mostly interested in the case that there is only one interesting degree. 
Given a simple root $i$ of the group and a sign character $\epsilon$ the space of $\LI$-invariants 
$$ \LI_{i}(\pi,\p)^{\epsilon} \subset \Hom^{\cont}(F_{\p}^{\ast},\Co)$$
is a subspace of codimension at least one.
Here $\Co$ denotes a large enough $p$-adic field.
If strong multiplicity one holds, its codimension is exactly one.
Whether a character of $F_\p^{\ast}$  belongs to the space of $\LI$-invariants is decided by certain maps between cohomology groups of $\p$-arithmetic subgroups with values in duals of (locally analytic) generalized Steinberg representations.
These maps are induced by cup products with one-extensions of the smooth generalized Steinberg representation corresponding to the simple root $i$ with the locally analytic Steinberg representations (see Section \ref{Extensions} for a description of these extensions due to Ding, cf.~\cite{Ding}).\\

As a first step, in section \ref{Principal} we show that one can replace generalized Steinberg representations by locally analytic principal series representations and the extension classes by infinitesimal deformations of these principal series.
As a consequence, we prove an automorphic analogue of the Colmez--Greenberg--Stevens formula (see Proposition \ref{CGS}).
Let us remind ourselves that the Colmez--Greenberg--Stevens formula (see for example Theorem 3.4 of \cite{Ding}) states that one can compute Fontaine--Mazur $\LI$-invariants of Galois representations by deforming them in $p$-adic families.
Our analogue states that one can compute automorphic $\LI$-invariants from the cohomology of $\p$-arithmetic groups with values in duals of big principal series representations, i.e.~parabolic inductions of characters with values in units of affinoid algebras.\\

Thus, we reduce the problem to producing classes in these big cohomology groups.
Here is where we need to impose further restrictions.
Firstly, we assume that the group under consideration is adjoint.
Under this assumption Kohlhaase and Schraen constructed a Koszul resolution of locally analytic principal series representations (cf.~\cite{KoSch}), which we recall in Section \ref{Koszul}.
Using that resolution we can lift overconvergent cohomology classes which are common eigenvectors for all $U_\p$-operators to big cohomology classes.
Secondly, in order to have a nice enough theory of families of overconvergent cohomology classes (in the spirit of \cite{AshSt}) we consider only groups whose Archimedean component fulfils the Harish-Chandra condition.
This implies that the automorphic representation we study only shows up in the middle degree cohomology of the associated locally symmetric space.
We further assume that the map from the eigenvariety to the weight space $\Ws$ is étale at the point corresponding to the automorphic representation $\pi$.
Étaleness is implied by a suitable strong multiplicity one assumption and, thus, holds for example for Hilbert modular forms. Under these hypotheses we can show that we can lift the cohomology class corresponding to $\pi$ to a big cohomology class valued in functions on an open affinoid neighbourhood of the trivial character in weight space (see Theorem \ref{thm:locfree}).
This allows us give the generalization of \eqref{basicformula} in Theorem \ref{thm:Linvderivative}.
In particular, we see that automorphic $\LI$-invariants are codimension one subspaces under our étaleness assumption.\\

As a first application, in Section \ref{JL} we prove a conjecture of Spie\ss{} (cf.~\cite{Sp}, Conjecture 6.4): we show that the $\LI$-invariants of Hilbert modular forms of parallel weight $2$ are independent of the sign character used to define them.
We further show that in this situation automorphic $\LI$-invariants are invariant under Jacquet--Langlands transfers to quaternion groups which are split at $\p$.
In fact, we show that all these $\LI$-invariants agree with the Fontaine--Mazur $\LI$-invariant of the associated Galois representation.\footnote{The same result has been obtained by Spie\ss~by different methods; see \cite{Sp3}.}
In particular, one can remove the assumptions in the main theorem of \cite{Sp} (see Theorem 6.10 (b) of \textit{loc.cit.}).
Furthermore, the equality of automorphic and Fontaine--Mazur $\LI$-invariants makes the construction of Stark--Heegner points for modular elliptic curves over totally real fields unconditional.
Similarly, we prove the equality of automorphic and Fontaine--Mazur $\LI$-invariants for definite unitary groups under mild assumptions in Section \ref{Galois}, {\it i.e.}~we give an alternative proof of the main result of \cite{Ding} for our global situation.\\

We end the paper by considering the easiest case of a group that does not fulfil the Harish-Chandra condition, i.e.~we study Bianchi modular forms.
As we cannot deform the automorphic representation over the whole weight space, in general one cannot compute the $\LI$-invariant completely in terms of derivatives of Hecke-eigenvalues.
But at least in case the Bianchi forms is the base change of a modular form, we overcome this problem;
we show that the $\LI$-invariants of the base change equal the $\LI$-invariant of the modular form.\\

The assumption that the coefficient system is trivial is not necessary for the arguments of this article.
But the construction of automorphic $\LI$-invariants relies on the existence of a well-behaved lattice in the Steinberg representation, which is not known for twists of the Steinberg by an algebraic representation in general. 
In case the existence of such a lattice is known, e.g.~if $G$ is a form of $\PGL_2$ by a result of Vignéras (see \cite{Vi}), one can easily generalize our results to arbitrary weights and non-critical slopes.

\subsection*{Notations}
If $X$ and $Y$ are topological spaces, we write $C(X,Y)$ for the set of continuous functions from $X$ to $Y$.
All rings are assumed to be commutative and unital.
The group of invertible elements of a ring $R$ will be denoted by $R^{\ast}$.
If $M$ is an $R$-module we denote the $q$-th exterior power of $M$ by $\Lambda^{q}_R M$.
If $R$ is a ring and $G$ a group, we will denote the group ring of $G$ over $R$ by $R[G]$.
Given topological groups $H$ and $G$ we write $\Hom^{\cont}(H,G)$ for the space of continuous homomorphism from $H$ to $G$.
Let $\chi\colon G\to R^{\ast}$ be a character.
We write $R[\chi]$ for the $G$-representation, which underlying $R$-module is $R$ itself and on which $G$ acts via the character $\chi$.
The trivial character of $G$ will be denoted by $\cf_G$.
Let $H$ be an open subgroup of a locally profinite group $G$ and $M$ an $R$-linear representation $M$ of $H$.
The \textit{compact induction} $\cind^{G}_{H}M$ of $M$ from $H$ to $G$ is the space of all functions $f\colon G\to M$ which have finite support modulo $H$ and satisfy $f(gh)=h^{-1}.f(g)$ for all $h\in H,\ g\in G$.

\subsection*{Acknowledgements}
While working on this manuscript the first named author was visiting McGill University, supported by Deutsche Forschungsgemeinschaft, and he would like to thank these institutions.
In addition, the authors would like to thank Henri Darmon, Michael Lipnowski, Vytautas Pa\u{s}k\={u}nas and Chris Williams for several intense and stimulating discussions. We also thank the referee for their careful reading of the paper and many suggestions.

%% file: Big-Setup.tex
\subsection*{The setup}
We fix an algebraic number field $F$.
In addition, we fix a finite place $\p$ of $F$ lying above the rational prime $p$ and choose embeddings
$$\overline{\Q_p} \hookleftarrow \overline{\Q}\hookrightarrow \C.$$

If $v$ is a place of $F$, we denote by $F_{v}$ the completion of $F$ at $v$.
If $v$ is a finite place, we let $\mathcal{O}_{v}$ denote the valuation ring of $F_{v}$ and $\ord_{v}$ the additive valuation such that $\ord_{v}(\varpi)=1$ for any local uniformizer $\varpi\in\mathcal{O}_{v}$.

Let $\A$ be the adele ring of $F$, i.e~the restricted product over all completions $F_{v}$ of $F$.
We write $\A^\infty$ (respectively $\A\pinfty$) for the restricted product over all completions of $F$ at finite places (respectively finite places different from $\p$).
More generally, if $S$ is a finite set of places of $F$ we denote by $\A^{S}$ the restricted product of all completions $F_v$ with $v\notin S$.

If $H$ is an algebraic group over $F$ and $v$ is a place of $F$, we write $H_v=H(F_v)$.
We put $H_\infty=\prod_{v\mid\infty}H_v$.

Throughout the article we fix a connected, adjoint, semi-simple algebraic group $G$ over $F$.
We assume that the base change $G_{F_\p}$ of $G$ to $F_\p$ is split.
Let $K_\infty\subseteq G_\infty$ denote a fixed maximal compact subgroup.
The integers $\delta$ and $q$ are defined via
\begin{align*}\delta&=\rk G_\infty - \rk K_\infty \\
\intertext{and}
2 q+\delta&=\dim G_\infty - \dim K_\infty.
\end{align*}

At last, we fix a cuspidal automorphic representation $\pi=\otimes_v \pi_v$ of $G(\A)$ with the following properties:
\begin{itemize}
\item $\pi$ is cohomological with respect to the trivial coefficient system,
\item $\pi$ is tempered at $\infty$ and
\item $\pi_\p$ is the (smooth) Steinberg representation $\St^{\infty}_{G_{\p}}(\C)$ of $G_\p$.
\end{itemize}
We denote by $\Q_\pi\subseteq\overline{\Q}$ a fixed finite extension of $\Q$ over which $\pi^{\infty}$ has a model (see Theorem C of \cite{FabianRat} for the existence of such an extension).

\begin{Hyp}[SMO]\label{Hyp1}
We assume that the following strong multiplicity one hypothesis on $\pi$ holds:
If $\pi^{\prime}$ is an automorphic representation of $G$ such that 
\begin{itemize}
\item $\pi_v^{\prime}\cong\pi_v$ for all finite places $v\neq\p$,
\item $\pi_\p^{\prime}$ has an Iwahori-invariant vector and
\item $\pi^{\prime}_\infty$ has non-vanishing $(\mathfrak{g},K_{\infty}^{\circ})$-cohomology,
\end{itemize}
then $\pi$ is cuspidal, $\pi^{\prime}_\p\cong\pi_\p$ and $\pi_\infty$ is tempered.
\end{Hyp}

It is known that this holds for cuspidal representations of $\mathrm{GL}_n$ by work of Jacquet--Piateski-Shapiro--Shalika \cite{JS1,JS2, PSMultiplicityOne} and thus in particular, for representations of $G=\mathrm{PGL}_n$.

For general groups strong multiplicity one fails, see for example \cite{HowePS}.
Hence, it is harder to find explicit results about representations $\pi$ for which the hypothesis SMO holds, but nevertheless  it is expected that SMO should hold in many cases; for example for  generic representation of $\mathrm{GSp}_4$ \cite{SoudryUnique}.
The same strategy of {\it loc. cit.} could apply as long as we have an injective transfer from (a subclass of) representations of $G$ to representation of $\mathrm{GL}_n$, e.g., tempered  representations of unitary groups that at each prime are not endoscopic.

%% file: Big-AL-Extensions.tex
\subsection{Extensions}\label{Extensions}
In this section we recall the computation of certain $\Ext^{1}$-groups of (locally analytic) generalized Steinberg representations due to Ding (see \cite{Ding}).
We fix a finite extension $\Co$ of $\Q_p$. If $V$ and $W$ are admissible locally $\Q_p$-analytic $\Co$-representations of $G_\p$, we write $\Ext^1_{\an}(V,W)$ for the group of locally analytic extensions of $V$ by $W$.

Given an algebraic subgroup $H\subseteq G_{F_\p}$, we denote the group of $F_\p$-valued points of $H$ also by $H$.
We fix a Borel subgroup $B$ of the split group $G_{F_\p}$ and a maximal split torus $T\subseteq B$ and denote by $\Delta$ the associated basis of simple roots.
For a subset $I\subseteq\Delta$ we let $P_I\supseteq B$ be the corresponding parabolic containing $B$.

Suppose $M$ is a smooth representation of $P_I$ over a ring $R$; we define its smooth induction to $G_\p$ as
$$i_{P_I}^{\infty}(M)=\left\{f\colon G_\p\to M \mbox{ locally constant}\mid f(pg)=p.f(g)\ \forall p\in P_I,\ g\in G_\p \right\}.$$
The generalized $R$-valued (smooth) Steinberg representation associated with $I\subseteq \Delta$ is given by the quotient
$$v^{\infty}_{P_I}(R)= i^{\infty}_{P_I}(R) / \sum_{I\subset J\subset \Delta, I\neq J} i^{\infty}_{P_J}(R).$$
Likewise, if $V$ is a locally $\Q_\p$-analytic $\Co$-representation of $P_I$, we define its locally analytic induction to $G_\p$ as the space of functions
$$\II_{P_I}^{\an}(V)=\left\{f\colon G_\p\to \tau \mbox{ locally } \Q_p\mbox{-analytic}\mid f(pg)=p.f(g)\ \forall p\in P_I,\ g\in G_\p \right\}.$$
We define the locally analytic generalized Steinberg representation with respect to $I$ as the quotient
$$\VV^{\an}_{P_I}(\Co)=\II_{P_I}^{\an}(\Co)/\sum_{I\subset J\subset \Delta, I\neq J} \II^{\an}_{P_J}(\Co).$$
We put $\St_{G_\p}^{\infty}(R)=v^{\infty}_B(R)$ and $\St_{G_\p}^{\an}(\Co)=\VV^{\an}_B(\Co)$.
Similarly, replacing locally analyticity with continuity we define the continuous Steinberg representation $\St_{G_\p}^{\cont}(\Co)$ and, more generally, $\VV^{\cont}_{P_I}(\Co)$.
It is easy to see that $\VV^{\cont}_{P_I}(\Co)$ is the universal unitary completion of both $v^{\infty}_{P_I}(\Co)$ and $\VV^{\an}_{P_I}(\Co)$.

Let $i\in\Delta$ be a simple root and $\lambda\in\Hom^{\cont}(B,\Co)$ a continuous homomorphism. Note that $\lambda$ is automatically locally analytic and is trivial on the unipotent radical of $B$. Thus it can be identified with a character on $T$. 
We write $\tau_\lambda$ for the the two-dimensional representation of $B$ given by
\begin{align}\label{tau}
\tau_\lambda(b)=\begin{pmatrix} 1 & \lambda(b)\\ 0 & 1\end{pmatrix}.
\end{align}
By the exactness of the parabolic induction functor for locally analytic extensions (see \cite{Kohlhaase}, Proposition 5.1 and Remark 5.4) we have a short exact sequence of the form
$$0\too \II_{B}^{\an}(\Co)\too \II_{B}^{\an}(\tau_\lambda)\too \II_{B}^{\an}(\Co)\too 0.$$
We write
\begin{align}\label{class1}
\widehat{\mathcal{E}}^{\an}(\lambda)\in \Ext^{1}_{\an}(\II^{\an}_{B}(\Co),\II_{B}^{\an}(\Co))
\end{align}
for the associated extension class. 
Further, we define the class
\begin{align}\label{class2}
\widetilde{\mathcal{E}}^{\an}_{i}(\lambda)\in \Ext^{1}_{\an}(i^{\infty}_{P_i}(\Co),\II_{B}^{\an}(\Co))
\end{align}
as the pullback of this extension along $i^{\infty}_{P_i}(\Co)\to \II^{\an}_{B}(\Co)$.
Finally, taking pushforward along $\II_{B}^{\an}(\Co)\to \St_{G_\p}^{\an}(\Co)$ yields the extension class
\begin{align}\label{class3}
\mathcal{E}^{\an}_{i}(\lambda)\in \Ext^{1}_{\an}(i^{\infty}_{P_i}(\Co),\St_{G_\p}^{\an}).
\end{align}

By an easy calculation we see that the map
$$\Hom^{\cont}(B,\Co)\too \Ext^{1}_{\an}(i^{\infty}_{P_i}(\Co),\St_{G_\p}^{\an}(\Co)),\ \lambda\mapstoo \mathcal{E}^{\an}_{i}(\lambda)$$
defines a homomorphism.
The inclusion $B\into P_i$ induces an injection
$$\Hom^{\cont}(P_i,\Co)\too \Hom^{\cont}(B,\Co).$$
The quotient can be identified with the space $\Hom^{\cont}(F_\p^{\ast},\Co)$ via the map
\begin{align}\label{root}
\Hom^{\cont}(F_\p^{\ast},\Co)\too \Hom^{\cont}(B,\Co)/\Hom^{\cont}(P_i,\Co),\ \lambda\mapstoo \lambda\circ i.
\end{align}
Alternatively, let $i^{\vee}$ denote the coroot associated with $i$.
Then, the kernel of the map
\begin{align}\label{coroot}
\Hom^{\cont}(B,\Co)\too \Hom^{\cont}(F_\p^{\ast},\Co),\ \lambda\mapstoo \lambda_{i}=\lambda\circ i^{\vee}
\end{align}
is equal to $\Hom^{\cont}(P_i,\Co)$ and, hence, the map induces an isomorphism
$$\Hom^{\cont}(B,\Co)/\Hom^{\cont}(P_i,\Co)\too \Hom^{\cont}(F_\p^{\ast},\Co),$$
which is the inverse of the isomorphism above up to multiplication by two. 
\begin{Thm}[Ding]\label{Ding}
The following holds:
\begin{enumerate}[(i)]
\item The map $\Hom^{\cont}(B,\Co)\to \Ext^{1}_{\an}(i^{\infty}_{P_i}(\Co),\St_{G_\p}^{\an}(\Co)),\ \lambda\mapsto \mathcal{E}^{\an}_{i}(\lambda)$ is surjective with kernel $\Hom^{\cont}(P_i,\Co)\subseteq \Hom^{\cont}(B,\Co)$.
\item The canonical map $\Ext^{1}_{\an}(v^{\infty}_{P_i}(\Co),\St_{G_\p}^{\an}(\Co))\to \Ext^{1}_{\an}(i^{\infty}_{P_i}(\Co),\St_{G_\p}^{\an}(\Co))$ is an isomorphism.
\item The induced map $\Hom^{\cont}(F_\p^{\ast},\Co)\too \Ext^{1}_{\an}(v^{\infty}_{P_i}(\Co),\St_{G_\p}^{\an}(\Co)),\ \lambda \mapsto \mathcal{E}^{\an}_{i}(\lambda\circ i)$ is an isomorphism.
\end{enumerate}
\end{Thm}
\begin{proof}
The third claim is a direct consequence of the first two.
For the proof of the first two claims in the case $G=\PGL_n$ see Section 2.2 of \cite{Ding}.
The general case is proven in Section 2.4 of \cite{Ge4}.
\end{proof}

%% file: Big-AL-Resolutions.tex
\subsection{Flawless lattices}\label{Resolutions}
We recall the notion of flawless smooth representations over a ring $R$.

\begin{Def}
A smooth $R$-representation $M$ of $G_\p$ is called flawless if 
\begin{itemize}
\item $M$ is projective as an $R$-module and
\item there exists a finite length exact resolution
$$0\too C_m\too\cdots\too C_0\too M \too 0$$
of $M$ by smooth $R$-representations $C_i$ of $G_\p$, where each $C_i$ is a finite direct sum of modules of the form
$$\cind_{K_\p}^{G_\p}(L)$$
with $K_\p\subseteq G_\p$ a compact, open subgroup and $L$ a smooth representation of $K_\p$ that is finitely generated projective over $R$.
\end{itemize}
\end{Def}

The following is our main example:
\begin{Thm}[Borel--Serre]\label{stflaw}
The Steinberg representation $\St_{G_\p}^{\infty}(R)$ is flawless for every ring $R$.
\end{Thm}
\begin{proof}
It is enough to prove that $\St_{G_\p}^{\infty}(\Z)$ is flawless.
By \cite{BS2}, Theorem 5.6., $\St_{G_\p}^{\infty}(\Z)$ can be identified with the cohomology with compact supports of the Bruhat-Tits building of $G_\p$.
Thus, its simplicial complex gives a flawless resolution of $\St_{G_\p}^{\infty}(\Z)$. 
\end{proof}

%% file: Big-AL-Cohomology.tex
\subsection{Cohomology of $\p$-arithmetic groups}\label{Cohomology}
Let $A$ be a $\Q_p$-affinoid algebra in the sense of Tate.
Given a compact, open subgroup $K^\p\subseteq G(\A\pinfty)$, an $A[G_\p]$-module $V$ and an $A[G(F)]$-module $W$, which is free and of finite rank over $A$, we define $\Fu_A(K^{\p},V;W)$ as the space of all $A$-linear maps $\Phi\colon G(\A\pinfty)/K^\p \times V\to W$.
The $A$-module $\Fu_\Co(K^\p,V;W)$ carries a natural $G(F)$-action given by
$$(\gamma.\Phi)(g,v)=\gamma.(\Phi(\gamma^{-1}g,\gamma^{-1}.v)).$$
Suppose $V$ is a topological $A$-module equipped with a continuous $A$-linear $G_\p$-action we put
$\Fu^{\cont}_A(K^\p,V;W)=C(G(\A\pinfty)/K^{\p},\Hom_{A}^{\cont}(V,W))$.
Here $W$ is endowed with its canonical topology as a finitely generated free $A$-module.

Let $\Co$ be a finite extension of $\Q_p$ with ring of integers $\mathcal{O}_\Co$.
Suppose that $V$ is a smooth $\Co$-representation of $G_\p$ that admits a flawless $G_\p$-stable $\mathcal{O}_\Co$-lattice $M$.
Since $M$ is finitely generated as an $\mathcal{O}_\Co[G_\p]$-module, the completion of $V$ with respect to $M$ is the universal unitary completion $V^{\univ}$ of $V$.
The following automatic continuity statement holds (see \cite{Ge4}, Proposition 3.11).
\begin{Pro}\label{automatic}
Suppose that $V$ is a smooth $\Co$-representation of $G_\p$ that admits a flawless $G_\p$-stable $\mathcal{O}_\Co$-lattice $M$.
Then the canonical map
$$\HH^{d}(G(F),\Fu^{\cont}_\Co(K^\p,V^{\univ};\Co(\epsilon)))\too \HH^{d}(G(F),\Fu_\Co(K^\p,V;\Co(\epsilon)))$$
is an isomorphism for every character $\epsilon\colon \pi_0(G_\infty)\to\left\{\pm 1\right\}$, every compact, open subgroup $K^\p\subseteq G(\A\pinfty)$ and every degree $d\geq 0$.
\end{Pro}

The proposition above combined with Theorem \ref{stflaw} implies the following:
\begin{Cor}\label{automatic2}
The canonical map
\begin{align*}\HH^{d}(G(F),\Fu^{\cont}_\Co(K^\p,\St_{G_\p}^{\cont}(\Co);\Co(\epsilon)))
&\too \HH^{d}(G(F),\Fu_\Co(K^\p,\St_{G_\p}^{\infty}(\Co);\Co(\epsilon)))
\end{align*}
is an isomorphism.
\end{Cor}

For a compact, open subset $K_\p\subseteq G_\p$ and a character $\epsilon\colon \pi_0(G_\infty) \to \left\{\pm 1\right\}$ we put
$$\HH^{d}(X_{K^\p\times K_\p},W)^{\epsilon}=\HH^{d}(G(F),\Fu_\Co(K^\p,\Co[G_\p/K_\p];W(\epsilon))).$$
If the level $K^\p\times K_\p$ is neat, this group is naturally isomorphic to (the epsilon component of) the singular cohomology with coefficients in $W$ of the locally symmetric space of level $K^\p\times K_\p$ associated with $G$.
More generally, let $V$ be a topological $A$-module, which is locally convex as a $\Q_p$-vector space, with a continuous $A$-linear $K_\p$-action.
We consider $\cind_{K_\p}^{G_\p}V$ with the locally convex inductive limit topology.
This is a continuous $A$-module with a continuous $G_p$-action.
We put
$$\HH^{d}(X_{K^\p\times K_\p},\Hom_{A}^{\cont}(V,W))^{\epsilon}=\HH^{d}(G(F),\Fu_{A}^{\cont}(K^\p,\cind_{K_\p}^{G_\p}V;W(\epsilon))).$$
Again, this space can be identified with the cohomology of the corresponding locally symmetric space with values in the sheaf associated with $\Hom_{A}^{\cont}(V,A)$.

%% file: Big-AL-Definition.tex
\subsection{Automorphic $\LI$-invariants}\label{Definition} 
For the remainder of the article we fix a finite extension $\Co$ of $\Q_p$ that contains $\Q_\pi$ and a compact, open subgroup $K^\p=\prod_{v\nmid \p \infty}K_v\subseteq G(\A\pinfty)$ such that $(\pi\pinfty)^{K^\p}\neq 0$.
We may assume that $K_v$ is hyperspecial for every finite place $v$ such that $\pi_v$ is spherical.

Since $v_{P_I}^{\infty}(\Co)=v_I^{\infty}(\Z)\otimes \Co$ we have a canonical isomorphism
$$\Fu_\Z(K^\p,v^{\infty}_{P_I}(\Z);W)\cong \Fu_{\Co}(K^\p,v^{\infty}_{P_I}(\Co);W)$$
for any $\Co$-vector space $W$.
Hence, we abbreviate this space by $\Fu(K^\p,v^{\infty}_{P_I};W)$ (and similarly for $\St_{G_\p}^{\infty}$ in place of $v^{\infty}_{P_I}$).

Let $I_\p\subseteq G_\p$ be an Iwahori subgroup.
By Frobenius reciprocity the choice of an Iwahori-fixed vector yields a $G_\p$-equivariant homomorphism
\begin{align}\label{evaluationzero}
\cind_{I_\p}^{G_\p}\Co\too \St_{G_\p}^{\infty},
\end{align}
which in turn induces a Hecke-equivariant map
\begin{align}\label{evaluation}
\ev^{(d)}\colon \HH^{d}(G(F),\Fu(K^\p,\St_{G_\p}^{\infty};\Co(\epsilon)))\too \HH^{d}(X_{K^\p\times I_\p},\Co)^{\epsilon}.
\end{align}

Let
$$\T=\T(K^\p\times I_\p)_\Co=C_c(K^\p\times I_\p\backslash G(\A^\infty)/K^\p\times I_\p,\Co)$$
be the $\Co$-valued Hecke algebra of level $K^\p\times I_\p$.
By abuse of notation we denote the model of $\pi^\infty$ over $\Co$ also by $\pi^\infty$.
If $V$ is a $\mathbb{\T}$-module, we put
$$V[\pi]=\sum_f \im(f)$$
where we sum over all $\mathbb{T}$-homomorphisms $f\colon(\pi^\infty)^{K^\p\times I_\p}\to V.$
Similarly, we define
$$\T^{\p}=\T(K^\p)_\Co=C_c(K^\p\backslash G(\A\pinfty)/K^\p,\Co)$$
to be the Hecke algebra away from $\p$ and, given a $\T^{\p}$-module $V$, we put
$$V[\pi^{\p}]=\sum_f \im(f)$$
where we sum over all $\T^{\p}$-homomorphisms $f\colon(\pi\pinfty)^{K^\p}\to V.$

For the proof of the following proposition that crucially relies on the hypothesis (SMO) we refer to \cite{Ge4}, Proposition 3.6.
\begin{Pro}\label{componentpro}
The map $\ev^{(d)}$ induces an isomorphism
$$\HH^{d}(G(F),\Fu(K^\p,\St_{G_\p}^{\infty};\Co(\epsilon)))[\pi^{\p}]\xrightarrow{\cong} \HH^{d}(X_{K^\p\times I_\p},\Co)^{\epsilon}[\pi].$$
on isotypic components.
There exists an integer $m_\pi\geq 0$ such that
$$\dim_\Co \HH^{d}(X_{K^\p\times I_\p},\Co)^{\epsilon}[\pi]= m_\pi\cdot \dim(\pi^\infty)^{K^\p\times I_\p} \cdot\binom{\delta}{d-q}\ \forall d\geq 0$$
for each sign character $\epsilon.$
\end{Pro}

By Corollary \ref{automatic2} we have a canonical isomorphism
$$\HH^{d}(G(F),\Fu(K^\p,\St_{G_\p}^{\infty};\Co(\epsilon))) \cong \HH^{d}(G(F),\Fu_\Co^{\cont}(K^\p,\St_{G_\p}^{\cont};\Co(\epsilon))).$$
Composing with the homomorphism coming from dualizing the continuous inclusion $\St_{G_\p}^{\an}\into\St_{G_\p}^{\cont}$ yields the map
$$
\HH^{d}(G(F),\Fu(K^\p,\St_{G_\p}^{\infty};\Co(\epsilon)))\too \HH^{d}(G(F),\Fu_\Co^{\cont}(K^\p,\St_{G_\p}^{\an};\Co(\epsilon)))
$$
in cohomology.
Thus, for every $i\in\Delta$ we get a well-defined cup-product pairing
\begin{align*}
&\HH^{d}(G(F),\Fu(K^\p,\St_{G_\p}^{\infty};\Co(\epsilon))) \times \Ext^{1}_{\an}(v_{P_i}^{\infty}(\Co),\St_{G_\p}^{\an}(\Co))\\
\too &\HH^{d+1}(G(F),\Fu(K^\p,v^{\infty}_{P_i};\Co(\epsilon)))
\end{align*}
which commutes with the action of the Hecke algebra $\T^{\p}$.
By Theorem \ref{Ding} we have a canonical isomorphism $\Hom^{\cont}(F_\p^{\ast},\Co)\cong\Ext^{1}_{\an}(v^{\infty}_{P_i}(\Co),\St_{G_\p}^{\an}(\Co))$.
Hence, taking cup product with the extension $\mathcal{E}^{\an}_{i}(\lambda\circ i)$ associated with a homomorphism $\lambda\in \Hom^{\cont}(F_\p^{\ast},\Co)$ in \eqref{class3} yields a map
$$c^{(d)}_{i}(\lambda)^{\epsilon}\colon \HH^{d}(G(F),\Fu(K^\p,\St_{G_\p}^{\infty};\Co(\epsilon)))[\pi^{\p}]\too \HH^{d+1}(G(F),\Fu(K^\p,v^{\infty}_{P_i};\Co(\epsilon)))[\pi^{\p}]$$
on $\pi\pinfty$-isotypic parts.

\begin{Def}\label{deficom}
Given a character $\epsilon\colon \pi_0(G_\infty) \to \left\{\pm 1\right\}$, an integer $d\in \Z$ with $0\leq d\leq \delta$ and a root $i\in\Delta$ we define
$$\LI_{i}^{(d)}(\pi,\p)^{\epsilon}\subseteq \Hom^{\cont}(F_\p^{\ast},\Co)$$
as the kernel of the map $\lambda \mapsto c^{(q+d)}_{i}(\lambda)^{\epsilon}$.
\end{Def}
Alternatively, we can consider cup products with extensions associated with homomorphisms $\lambda\colon B \to \Co$ and define the $\LI$-invariant as a subspace of
$$\Hom^{\cont}(B,\Co)/\Hom^{\cont}(P_i,\Co).$$
This subspace is mapped to the $\LI$-invariant defined above via the map \eqref{coroot} induced by the coroot $i^{\vee}$ associated with $i$.

\begin{Pro}\label{codimension}
For all sign characters $\epsilon$, every degree $d\in [0,\delta]\cap \Z$ and every root $i \in \Delta$ the $\LI$-invariant
$$\LI_{i}^{(d)}(\pi,\p)^{\epsilon}\subseteq \Hom^{\cont}(F_\p^{\ast},\Co)$$
is a subspace of codimension at least one, which does not contain the space of smooth homomorphisms.

Suppose $m_\pi=1$.
Then, in the extremal cases $d=0$ and $d=\delta$ the codimension is exactly one.
\end{Pro}
\begin{proof}
This is Proposition 3.14 of \cite{Ge4}.
\end{proof}

\begin{Rem}
Let $\log_p\colon \Co^{\ast}\to \Co$ denote the branch of the $p$-adic logarithm such that $\log_p(p)=0.$
Let us assume that $\Co$ contains the images of all embeddings $\sigma\colon F_\p\to \overline{\Q_p}.$
We put $\log_{p,\sigma}=\log_p \circ \sigma\colon F_\p^{\ast}\to \Co.$
The set
$$\left\{\ord_\p\right\} \cup \left\{\log_{p,\sigma}\mid \sigma\colon F_\p \to \Co \right\}$$
is a basis of $\Hom^{\cont}(F_\p^{\ast},\Co).$
Suppose $\LI\subseteq \Hom^{\cont}(F_\p^{\ast},\Co)$ is a subspace of codimension one that does not contain $\ord_\p$.
Then for each embedding $\sigma$ there exists a unique element $\LI^{\sigma}\in \Co$ such that $\log_{p,\sigma}- \LI^{\sigma}\ord_\p \in \LI$ and these elements clearly form a basis of $\LI$. 
\end{Rem}

%% file: Big-Ref-Dualnumbers.tex
\subsection{Linear algebra over the dual numbers}\label{Dual numbers}
Let $R$ be a ring and $R[\varepsilon]=R[X]/X^2$ the ring of dual numbers over $R$.
Let $M$ be an $R[\varepsilon]$-module.
We define several maps on dual spaces associated with $M$.
First, multiplication with $\varepsilon$ induces the map $\mu\colon M/\varepsilon M\to \varepsilon M\hookrightarrow M$.
We denote by
$$\mu_\varepsilon^{\ast}\colon \Hom_R(M,R)\too \Hom_R(M/\varepsilon M,R)$$
its $R$-dual.
Second, reducing modulo $\varepsilon$ yields the map
$$\red\colon \Hom_{R[\varepsilon]}(M,R[\varepsilon])\too \Hom_{R}(M/\varepsilon M,R).$$
At last, the map $\add\colon R[\varepsilon]\to R, a+b\varepsilon\mapsto a+b$ induces the map
$$\add_\ast\colon\Hom_{R[\varepsilon]}(M,R[\varepsilon])\too \Hom_{R}(M,R).$$

The following easy computation is left to the reader.
\begin{Lem}\label{duallemma}
The maps $\mu_\varepsilon^{\ast}$, $\red$ and $\add$ are functorial in $M$ and the following equality holds: $\red=\mu_\varepsilon^{\ast}\circ\add_\ast.$
\end{Lem}

%% file: Big-Ref-Principalseries.tex
\subsection{Infinitesimal deformations of principal series}\label{Principal}
Let $T\subseteq B\subset G_{F_\p}$ be the maximal torus respectively the Borel subgroup chosen in Section \ref{Extensions}.
The map
$$\tau\colon\Hom^{\cont}(B,\Co)\too\Hom^{\cont}(B,\Co[\varepsilon]^{\ast}),\ \lambda \mapstoo \left[x\mapsto 1+\lambda(x)\epsilon\right]$$
defines an injective group homomorphism.
Its image is the set of all continuous characters $\chi\colon B\to \Co[\varepsilon]^{\ast}$ such that $\chi \equiv 1 \bmod \varepsilon$.
The underlying $\Co$-representation of $\tau(\lambda)$ is the two-dimensional representation $\tau_\lambda$ defined in \eqref{tau}.
Thus, we can view $\II_{B}^{\an}(\tau_\lambda)=\II_{B}^{\an}(\tau(\lambda))$ as an $\Co[\varepsilon]$-representation of $G_\p$.
Reducing modulo $\varepsilon$ induces the map
$$\red_\lambda^{d,\epsilon}\colon \HH^{d}(G(F),\Fu^{\cont}_{\Co[\varepsilon]}(K^\p,\II_{B}^{\an}(\tau_\lambda);\Co[\varepsilon](\epsilon)))\too\HH^{d}(G(F),\Fu^{\cont}_\Co(K^\p,\II_{B}^{\an}(\Co);\Co(\epsilon)))$$
in cohomology. 

Let $i^\vee$ be the coroot associated with $i$.
Given a character $\lambda\colon B\to \Co$ we put $\lambda_i=\lambda \circ i^{\vee}$.

\begin{Lem}\label{difflemma}
Let $\lambda\colon B\to \Co$ be a continuous character.
If the isotypic component $\HH^{q+d}(G(F),\Fu^{\cont}_\Co(K^\p,\II_{B}^{\an}(\Co);\Co(\epsilon)))[\pi^{\p}]$ is contained in the image of $\red_\lambda^{q+d,\epsilon}$, then the homomorphism $\lambda_i$ belongs to $\LI_{i}^{(d)}(\pi,\p)^{\epsilon}.$
\end{Lem}
\begin{proof}
The inclusion $\II_{B}^{\an}(\Co)\into\II_{B}^{\an}(\tau_\lambda)$ induces the map
$$\HH^{q+d}(G(F),\Fu^{\cont}_{\Co}(K^\p,\II_{B}^{\an}(\tau_\lambda);\Co(\epsilon)))\too\HH^{q+d}(G(F),\Fu^{\cont}_\Co(K^\p,\II_{B}^{\an}(\Co);\Co(\epsilon)))$$
in cohomology.
Its image is the kernel of the cup product with the extension class $\widehat{\mathcal{E}}^{\an}(\lambda)$ defined in \eqref{class1}.

Thus, by Lemma \ref{duallemma} our assumption implies that the $\pi$-isotypic component
$\HH^{q+d}(G(F),\Fu^{\cont}_\Co(K^\p,\II_{B}^{\an}(\Co);\Co(\epsilon)))[\pi^{\p}]$
is contained in the kernel of the cup product with $\widehat{\mathcal{E}}^{\an}(\lambda)$.
We have the following commutative diagram:
\begin{center}
\begin{tikzpicture}
    \path 	
		(7,3) 	node[name=E]{$\HH^{q+d+1}(G(F),\Fu^{\cont}_\Co(K^\p,\II_{B}^{\an}(\Co);\Co(\epsilon)))$}
		(0,3) 	node[name=2]{$\HH^{q+d}(G(F),\Fu^{\cont}_\Co(K^\p,\II_{B}^{\an}(\Co);\Co(\epsilon)))$}
		(0,1.5) 	node[name=A]{$\HH^{q+d}(G(F),\Fu^{\cont}_\Co(K^\p,\II_{B}^{\an}(\Co);\Co(\epsilon)))$}
		(7,1.5) 	node[name=B]{$\HH^{q+d+1}(G(F),\Fu_\Co(K^\p,i^{\infty}_{P_i}(\Co);\Co(\epsilon)))$}
		(0,0) 	node[name=C]{$\HH^{q+d}(G(F),\Fu^{\cont}_\Co(K^\p,\St_{G_\p}^{\cont}(\Co);\Co(\epsilon)))$}
		(7,0) 	node[name=D]{$\HH^{q+d+1}(G(F),\Fu_\Co(K^\p,v^{\infty}_{P_i}(\Co);\Co(\epsilon)))$};
    \draw[->] (C) -- (A) ;
		 \draw[->] (2) -- (A) node[midway, left]{$=$};
    \draw[->] (A) -- (B) node[midway, above]{$\cup \widetilde{\mathcal{E}}^{\an}_{i}(\lambda)$};
    \draw[->] (C) -- (D) node[midway, above]{$\cup \mathcal{E}^{\an}_{i}(\lambda)$};
    \draw[->] (D) -- (B) ;
		\draw[->] (2) -- (E) node[midway, above]{$\cup \widehat{\mathcal{E}}^{\an}(\lambda)$};
		\draw[->] (E) -- (B) ;
\end{tikzpicture} 
\end{center}
The claim now follows from the first part of the next lemma and a simple diagram chase.
\end{proof}

\begin{Lem}\label{injectivity}
Let $J\subseteq\Delta$ be a subset.
\begin{enumerate}[(i)]
\item The canonical map
$$ \HH^{d}(G(F),\Fu_\Co(K^\p,v^{\infty}_{P_J}(\Co);\Co(\epsilon)))[\pi^{\p}]\too\HH^{d}(G(F),\Fu_\Co(K^\p,i^{\infty}_{P_J}(\Co);\Co(\epsilon)))[\pi^{\p}]$$
is injective for all $d$. 
\item It is an isomorphism in degree $d=q+|J|$.
\end{enumerate}
\end{Lem}
\begin{proof}
The Jordan--Hölder decomposition of $i^{\infty}_{P_J}(\Co)$ consists of all generalized Steinberg representations $v^{\infty}_{P_I}(\Co)$ with $J\subseteq I$, each occurring with multiplicity one.
Thus, the second claim follows from \cite{Ge4}, Proposition 3.9.

Via the well-known resolution (see for example \cite{Orlik}, Proposition 11)
$$0 \to i^{\infty}_{P_\Delta}(\Co)\to \bigoplus_{\substack{I\subseteq J\subseteq \Delta\\ |\Delta\setminus I|=1}}i^{\infty}_{P_I}(\Co)
\to \ldots \to \bigoplus_{\substack{J\subseteq I\subseteq \Delta\\ |I\setminus J|=1}}i^{\infty}_{P_I}(\Co) \to i^{\infty}_{P_J}(\Co) \too v^{\infty}_{P_J}(\Co) \to 0 $$
one can reduce the first claim to the following statement:
Let $\mathcal{E}_{J,I}$ be any smooth extension of $v^{\infty}_{P_J}(\Co)$ by $v^{\infty}_{P_I}(\Co)$, where $J\subseteq I \subseteq \Delta$ with $|I|=|J|+1$.
Then the map
$$\HH^{d}(G(F),\Fu_\Co(K^\p,v^{\infty}_{P_J}(\Co);\Co(\epsilon)))[\pi^{\p}]\too\HH^{d}(G(F),\Fu_\Co(K^\p,\mathcal{E}_{J,I};\Co(\epsilon)))[\pi^{\p}]$$
is injective for all $d$.
Equivalently, it is enough to show that the cup product
$$\HH^{d}(G(F),\Fu_\Co(K^\p,v^{\infty}_{P_I}(\Co);\Co(\epsilon)))[\pi^{\p}]\xrightarrow{\cup \mathcal{E}_{J,I}}\HH^{d+1}(G(F),\Fu_\Co(K^\p,v^{\infty}_{P_J}(\Co);\Co(\epsilon)))[\pi^{\p}]$$ is the zero map for all $d$.
Let $\mathcal{E}_{I,J}$ be the unique up to scalar non-split smooth extension of $v^{\infty}_{P_I}(\Co)$ by $v^{\infty}_{P_J}(\Co)$.
By \cite{Ge4}, Corollary 3.8, the cup product
$$\HH^{d+1}(G(F),\Fu_\Co(K^\p,v^{\infty}_{P_J}(\Co);\Co(\epsilon)))[\pi^{\p}]\xrightarrow{\cup \mathcal{E}_{I,J}}\HH^{d+2}(G(F),\Fu_\Co(K^\p,v^{\infty}_{P_I}(\Co);\Co(\epsilon)))[\pi^{\p}]$$
is an isomorphism.
Therefore, it is enough to prove that taking the cup product with $\mathcal{E}_{J,I} \cup \mathcal{E}_{I,J}$ induces the zero map on $\HH^{\bullet}(G(F),\Fu_\Co(K^\p,v^{\infty}_{P_I}(\Co);\Co(\epsilon)))[\pi^{\p}].$
This is true since $\mathcal{E}_{J,I} \cup \mathcal{E}_{I,J}$ is a smooth $2$-extension of $v^{\infty}_{P_I}(\Co)$ by itself and the space of all such extensions is zero by \cite{Orlik}, Theorem 1.
\end{proof}


%% file: Big-Ref-Bigprincipalseries.tex
\subsection{The automorphic Colmez--Greenberg--Stevens formula}\label{Big principal series}
Let $A$ be an $\Co$-affinoid algebra and $\chi\colon B\to A^{\ast}$ a locally analytic character.
The parabolic induction $\II_B^{\an}(\chi)$ is naturally an $A[G_\p]$-module.
Given an ideal $\m\subseteq A$ we let $\chi_\m\colon B\to (A/\m)^{\ast},\ x\mapstoo \chi(x) \bmod \m$ denote the reduction of $\chi$ modulo $\m$.
Similar as in Proposition 2.2.1 of \cite{Hansen} one can prove that
\begin{align*}
\Hom_{A}^{\cont}(\II_B^{\an}(\chi),A)\otimes_A A/\m &\cong \Hom_{A/\m}^{\cont}(\II_B^{\an}(\chi_\m),A/\m).
\end{align*}
Let $\m\in \Spm A$ be an $\Co$-rational point such that $\chi_\m=\cf_{G\p}$.
Thus, the reduction map induces the maps
$$\red_\chi^{d,\epsilon}\colon \HH^{d}(G(F),\Fu^{\cont}_{A}(K^\p,\II_{B}^{\an}(\chi);A(\epsilon)))\too\HH^{d}(G(F),\Fu^{\cont}_\Co(K^\p,\II_{B}^{\an}(\Co);\Co(\epsilon)))$$
in cohomology.

Let $i^\vee$ be the coroot associated with $i$.
We put $\chi_i=\chi \circ i^{\vee} \in \Hom(F_{\p}^{\ast},A).$
Suppose $v\colon \Spec \Co[\varepsilon]\to \Spm A$ is an element of the tangent space of $\Spm A$ at $\m$.
The pullback $\chi_{i,v}$ of $\chi_i$ along $v$ is of the form $\chi_{i,v}=1+\frac{\partial}{\partial v}\chi_{i}\cdot \varepsilon$ for a unique homomorphism $\frac{\partial}{\partial v}\chi_{i}\colon F_\p^{\ast}\to \Co.$

Lemma \ref{difflemma} immediately implies the following:
\begin{Pro}\label{CGS}
Suppose that the image of $\red_\chi^{q+d,\epsilon}$ contains the $\pi\pinfty$-isotypic component of $\HH^{q+d}(G(F),\Fu^{\cont}_\Co(K^\p,\II_{B}^{\an}(\Co);\Co(\epsilon)))$.
Then the homomorphism $\frac{\partial}{\partial v}\chi_{i}$ belongs to $\LI_{i}^{(d)}(\pi,\p)^{\epsilon}$ for every element $v$ of the tangent space of $\Spm A$ at $\m$.
\end{Pro}

%% file: Big-OC-Koszul.tex
\subsection{Koszul complexes}\label{Koszul}
In \cite{KoSch} Kohlhaase and Schraen construct a resolution of locally analytic principal series representations via a Koszul complex.
We recall their construction in a slightly more general setup:
instead of restricting to $p$-adic fields as coefficient rings we allow affinoid algebras.
The proofs of \textit{loc.cit.}~carry over verbatim to this more general framework. 

Let us fix some notation:
we denote the Borel opposite of $B\subseteq{G_{F_\p}}$ by $\overline{B}$.
Let $\overline{N}\subseteq\overline{B}$ be its unipotent radical.
The chosen torus $T\subseteq B\subseteq G_{F_\p}$ determines an apartment in the Bruhat-Tits building of $G_\p$.
We chose a chamber $C$ of that apartment and a special vertex $v$ of $C$ as in Section 3.5. of \cite{CartierCorvallis}.
The stabilizer $G_{\p,0}\subseteq G_\p$ of $v$ is a maximal compact subgroup of $G_\p$ and the stabilizer $I_\p\subseteq G_{\p,0}$ of $C$ is an Iwahori subgroup.
Let $\mathfrak{G}_0$ be the Bruhat-Tits group scheme over $\mathcal{O}_\p$ associated with $G_{\p,0}$.
We define
$$I_\p^{n}=\Ker(I_\p\too \mathfrak{G}_0(\mathcal{O}_\p/\p^{n})).$$
The open normal subgroups $I_\p^{n}\subseteq I_\p$ form a system of neighbourhoods of the identity in $I_\p$.
The subgroup $T_0=T\cap G_{\p,0}$ is maximal compact subgroup of $T$.

Let $X^\ast(T)$ (respectively $X_{\ast}(T)$) denote the group of $F_\p$-rational characters (respectively cocharacters) of $T$.
The natural pairing
$$\left\langle \cdot,\cdot \right\rangle\colon X^\ast(T) \times X_\ast(T)\too \Z$$
is a perfect pairing.
There is a natural isomorphism $T/T_0\cong X_{\ast}(T)$ characterized by
$$\left\langle \chi,t\right\rangle=\ord_{\p}(\chi(t)).$$
We denote by $\Phi^{+}$ the set of positive roots with respect to $B$ and put
$$T^{-}=\left\{t\in T\mid \ord_{\p}(\alpha(t))\leq 0\ \forall \alpha\in\Phi^{+}\right\}.$$

Let $A$ be a $\Q_p$-affinoid algebra.
Restricting to $T_0$ gives a bijection between locally analytic characters $\chi\colon B\cap I_\p\to A^{\ast}$ and locally analytic characters $\chi\colon T_0\to A^{\ast}$.
Given such a character $\chi$ we write
$$\An_{\chi}=\left\{f\colon I_\p\too A\ \mbox{locally analytic}\mid f(bk)=\chi(b)f(k)\ \forall b\in B\cap I_\p,\ k\in I_p\right\}$$
for the locally analytic induction of $\chi$ to $I_\p$.
It is naturally an $A[I_\p]$-module.
Restricting a function $f\in \An_{\chi}$ to the intersection $I_\p\cap \overline{N}$ induces an isomorphism of $\An_\chi$ with the space of all locally analytic functions from $I_\p\cap \overline{N}$ to $A$.
There exists a minimal integer $n_\chi\geq 1$ such that $\chi$ restricted to $B\cap I_\p^{n_\chi}$ is rigid analytic.
For any $n\geq n_\chi$ we define the $A[I_\p]$-submodule
$$\An_{\chi}^{n}=\left\{f\in \An_{\chi} \mid f \mbox{ is rigid analytic on any coset in } I_\p/I_\p^{n}\right\}.$$

For later purposes we define the dual spaces
\begin{align*}
\Di^n_{\chi}&=\Hom_A^{\cont}(\An^n_{\chi}, A).
\end{align*}

By Frobenius reciprocity we can identify $\End_{A[ G_\p]}\left(\cind_{I_\p}^{G_\p} (\An^{n}_{\chi})\right)$ with the space of all functions $\Psi\colon G_\p\to \End_{A}(\An_{\chi}^{n})$ such that
\begin{itemize}
\item $\Psi$ is $I_\p$-biequivariant, i.e.~$\Psi(k_1 g k_2)=k_1 \Psi (g)k_2$ for all $k_1 ,k_2\in I_\p$, $g\in G_\p$, and 
\item for every $f\in\An_{\chi}^{n}$ the function $G_\p\to \An_{\chi}^{n},\ g\mapsto \Psi(g)(f)$ is compactly supported.
\end{itemize}

Let $t$ be an element of $T^{-}$ and $f\in \An_{\chi}^{n}.$
The function $I_\p\to A,\ u\mapsto f(tut^{-1})$ defines an element of $\An_{\chi}^{n}$.
\begin{Lem}\label{commutative}
For every element $t\in T^{-}$ there exists a unique $I_\p$-biequivariant function $\Psi_t\colon G_\p \to \End_A(\An_{\chi}^{n})$ such that
\begin{itemize}
\item $\supp(\Psi_t)= I_\p t^{-1} I_\p$ and
\item $\Psi_t(t^{-1})(f)(u)=f(tut^{-1})$ for any $f\in \An_{\chi}^{n}$ and $u\in I_\p\cap \overline{N}$.
\end{itemize}
\end{Lem}
\begin{proof}
This is a minor generalization of \cite{KoSch}, Lemma 2.2.
\end{proof}

For $t\in T^{-}$ we denote by $U_t$ the endomorphism of $\cind_{I_\p}^{G_\p} (\An^{n}_{\chi})$ corresponding to $\Psi_t$.
The following is a straightforward generalization of \cite{KoSch}, Lemma 2.3.
\begin{Lem}
We have $U_t U_{\tilde{t}}= U_{t\tilde{t}}$ for all $t,\tilde{t}\in T^{-}$.
\end{Lem}

Now let us fix a character $\chi\colon T\to A^{\ast}$ and let $\chi_0$ be its restriction to $T_0$.
Given an open subset $C\subseteq G_\p$, which is stable under multiplication with $B$ from the left, we denote by
$$\II^{\an}_B(\chi) (C)\subseteq \II^{\an}_B(\chi)$$
the subset of all functions with support in $C$.
Restricting functions to $I_\p$ gives an $I_\p$-equivariant $A$-linear isomorphism
$$\II^{\an}_B(\chi)(BI_\p)\xrightarrow{\cong} \An_{\chi_0}.$$
Thus, by Frobenius reciprocity its inverse induces a $G_\p$-equivariant $A$-linear map
\begin{align}\label{Koszulaug}
\aug_\chi\colon \cind_{I_\p}^{G_\p}(\An^{n}_{\chi_0})\too\cind_{I_\p}^{G_\p}(\An_{\chi_0})\too \II^{\an}_B(\chi).
\end{align}
for any integer $n\geq n_{\chi_0}$.

Since the group $G_{F_\p}$ is adjoint, there exist elements $t_i\in T^{-}$, $i\in\Delta$, such that
$$t_i\in \bigcap_{j\in\Delta\setminus\left\{i\right \}}\Ker(j)$$
and
$$\ord_{\p}(i(t_i))=-1.$$
The element $t_i$ is uniquely determined by the value $i(t_i)^{-1}\in F_\p^{\ast}$, which is a uniformizer.
Every element $t\in T$ can uniquely be written as $t=t_0\prod_{i\in\Delta}t_{i}^{n_i}$ with $t\in T_0$ and integers $n_{i}\in\Z$. 
Let us fix a choice of $t_i$, $i\in\Delta$, and put
$$y_i=U_{t_i}-\chi(t_i).$$ 

\begin{Pro}\label{Koszulkernel}
The $G_\p$-equivariant $A$-linear map $\aug_\chi$ is surjective with kernel $\sum_{i\in\Delta}\im(y_i).$
\end{Pro}
\begin{proof}
The same proof as for \cite{KoSch}, Proposition 2.4, works.
\end{proof}

By Lemma \ref{commutative} the $G_\p$-representation $\cind_{I_\p}^{G_\p}(\An^{n}_{\chi_0})$ is a module over the polynomial algebra $A[X_i\mid i\in\Delta]$, where $X_i$ acts through the operator $T_i$.
The Koszul complex of $\cind_{I_\p}^{G_\p}(\An^{n}_{\chi_0})$ with respect to the endomorphisms $(y_i)_{i\in\Delta}$ is the complex $\Lambda^{\bullet}_A(A^{\Delta}) \otimes \cind_{I_\p}^{G_\p}(\An_{\chi_0})$ with boundary maps
\begin{align}\label{boundary}
d_k(e_{i_1}\wedge\ldots\wedge e_{i_k}\otimes f)=\sum_{l=1}^{k}(-1)^{l+1} e_{i_1}\wedge\ldots\wedge \widehat{e_{i_l}}\wedge\ldots\wedge e_{i_k}\otimes y_{i_l}(f).
\end{align}

The following is the main technical theorem of Kohlhaase--Schraen (cf.~\cite{KoSch}, Theorem 2.5) generalized to affinoid coefficient rings.
\begin{Thm}[Kohlhaase--Schraen]\label{Koszulres}
For any $n\geq n_{\chi_0}$ the augmented Koszul complex
$$\Lambda^{\bullet}_A(A^{\Delta}) \otimes_A \cind_{I_\p}^{G_\p}(\An^{n}_{\chi_0})\too \II^{\an}_B(\chi)\too 0$$
with boundary maps \eqref{boundary} and augmentation map \eqref{Koszulaug} is exact.
\end{Thm}

\begin{Rem}
All of the results above remain valid if one replaces $\An_{\chi_0}^{n}$ by $\An_{\chi_0}$.
\end{Rem}

\begin{Exa}
Suppose $G_\p=\PGL_n(F_\p)$, $T$ is the torus of diagonal matrices and $B$ the Borel subgroup of upper triangular matrices.
The simple roots of $T$ with respect to $B$ are given by
$$i(\diag(x_1,\ldots,x_n))=x_i x_{i+1}^{-1}$$
for $1\leq i \leq d-1$.
For each simple root $1\leq i \leq d-1$ we might choose for $t_i$ the image of the diagonal matrix
$$\diag(1,\ldots,1,\pi,\ldots,\pi),$$
where $\pi$ is a uniformizer and exactly the first $i$ entries equal to one.
For the Iwahori subgroup $I_\p$ we may choose the image in $G_\p$ of all matrices in $\GL_n(\mathcal{O}_\p)$, which are upper triangular modulo $\p$.
Then $I_\p^{n}$ consists of all matrices in $I_\p$ which are congruent to the identity modulo $\p^{n}$.
\end{Exa}

There is also a smooth variant of the above result, which is probably well-known.
As we could not find a reference in the literature we give a proof of said variant in the following.

Let $\Omega$ be field of characteristic zero.
With the same formula as before, we can define commuting Hecke operators $U_t\in \End_{G_\p}(\cind_{I_\p}^{G_\p}(\Co))$ for $t\in T^{-}$.
Let $\cf\colon T_0\to \Omega^\times$ denote the constant character on $T_0$.
In case $\Omega=\Co$ we can identify $\Co$ with the subspace of constant functions in $\An^{n}_{\cf}$ and the induced embedding
\begin{align}\label{Iwahoriinclusion}
\cind_{I_\p}^{G_\p}(\Co)\into \cind_{I_\p}^{G_\p}(\An^{n}_{\cf})
\end{align}
is equivariant with respect to the operators $U_t$, $t\in T^{-}.$

The operators $U_t\in \End_{G_\p}(\cind_{I_\p}^{G_\p}(\Omega))$ are invertible.
Moreover, by the Bernstein decomposition the map
$$\Omega[X_i^{\pm 1} | i\in \Delta]\too \End_{G_\p}(\cind_{I_\p}^{G_\p}(\Omega)),\ X_i\mapstoo U_{t_i}$$
is injective and $\End_{G_\p}(\cind_{I_\p}^{G_\p}(\Omega))$ is a free-module of finite rank (and therefore flat) over $\Omega[X_i^{\pm 1} | i\in \Delta]$ (see for example \cite{ReederIwahori} for more details).
Since $\cind_{I_\p}^{G_\p}(\Omega)$ is a flat module over its endomorphism algebra by a theorem of Borel (see \cite{BorelIwahori}, Theorem 4.10), it is thus also flat as a $\Omega[X_i^{\pm 1} | i\in \Delta]$-module.

Therefore, for any choice of elements $a_i\in \Omega^{\times}$ the Koszul complex
$$\Lambda^{\bullet}_\Omega(\Omega^{\Delta}) \otimes_\Omega \cind_{I_\p}^{G_\p}(\Omega)$$
associated with the regular sequence $y_i=U_{t_i}-a_i$, $i\in \Delta$, is a resolution of the $G_\p$-representation
$$ M_{\underline{a}}=\Coker \left(\bigoplus_{i \in \Delta} \cind_{I_\p}^{G_\p}(\Omega) \xrightarrow{(y_{i})_{i\in \Delta}}\cind_{I_\p}^{G_\p}(\Omega)\right).$$

Let $\chi_{\underline{a}}\colon T\to \Omega^{\ast}$ be the unique smooth unramified character such that $\chi_{\underline{a}}(t_i)=a_i.$
As before, we extend $\chi_{\underline{a}}$ to a character of $B$.
Let $\phi\in i^{\infty}_B(\chi_{\underline{a}})$ be the unique element such that
\begin{itemize}
\item $\phi$ is invariant under $I_\p$,
\item the support of $\phi$ is $BI_\p$ and
\item $\phi(1)=1$.
\end{itemize}
Note that in case $\chi_{\underline{a}}=\cf$ is the trivial character the image of $\phi$ under the quotient map $i^{\infty}_B(\Omega)\onto \St^{\infty}(\Omega)$ is non-zero and, thus, generates the space of Iwahori-invariants of the Steinberg representation.

By Frobenius reciprocity $\phi$ induces a $G_\p$-equivariant homomorphism
\begin{align}\label{smoothaug}
\cind_{I_\p}^{G_\p}\Omega \too i^{\infty}_B(\chi_{\underline{a}}).
\end{align}
One can argue as in the proof of \cite{Ol}, Proposition 4.4, that the map \eqref{smoothaug} induces an isomorphism
$$M_{\underline{a}}\too i^{\infty}_B(\chi_{\underline{a}}).$$

In conclusion, we see that the augmented Koszul complex
\begin{align}\label{Koszulsmooth}
\Lambda^{\bullet}_A(A^{\Delta}) \otimes_A \cind_{I_\p}^{G_\p}(\Omega)\too i^{\infty}_B(\chi_{\underline{a}})\too 0
\end{align}
is exact.

%% file: Big-OC-Overconvergent.tex
\subsection{Overconvergent cohomology}\label{Overconvergent}
We show that the result from the previous section together with the theory of overconvergent cohomology allows us to construct cohomology classes with values in duals of locally analytic representations.

Before sticking to the case which is most relevant for our applications let us consider the general case: let $A$ be a $\Q_p$-affinoid algebra and $\chi\colon T\to A$ a locally analytic character.
Denote by $\chi_0$ its restriction to $T_0$ and fix elements $t_i\in T^{-}$, $i\in \Delta$, as before. 
For $n\geq n_{\chi_0}$ the continuous dual of the augmentation map \eqref{Koszulaug} of the previous section yields the map
$$\aug_\chi^{\ast}\colon\HH^{d}(G(F),\Fu^{\cont}_\Co(K^\p,\II^{\an}_{B}(\chi);A(\epsilon)))\too \HH^{d}(X_{K^\p\times I_\p},\Di^{n}_{\chi_0})^{\epsilon}$$
for every sign character $\epsilon$ and every integer $d\geq 0$.
We denote the operator induced by $U_{t_i}$, $i\in\Delta$, on the right hand side also by $U_{t_i}$ and similar for $U_{\tilde{t}}$.
Then, Proposition \ref{Koszulkernel} implies that
\begin{Cor}
We have:
$$\im (\aug_\chi^{\ast})\subseteq \bigcap_{i \in \Delta}\Ker (U_{t_i}-\chi(t_i)).$$
\end{Cor}

For the remainder of the section we study the case of the trivial character $\cf=\cf_{T}\colon T\to \Co^{\ast}$ of $T$ with values in the units of the $p$-adic field $\Co$.
We denote its restriction to $T_0$ also by $\cf$.
Let $n\geq 1$ be any integer.
The map \eqref{Iwahoriinclusion} induces a homomorphism
\begin{align}\label{classical}
\HH^{d}(X_{K^\p\times I_\p},\Di^{n}_{\cf})^{\epsilon}
\too \HH^{d}(X_{K^\p\times I_\p},\Co)^{\epsilon}
\end{align}
in cohomology, that is equivariant with respect to the commuting actions of the Hecke-algebra $\T^{\p}$ and the operators $U_t$, $t \in T^{-}$.

In the following we study the finite slope parts of the above cohomology groups (see for example \cite{Hansen}, Section 2.3, for definitions and notations).

\begin{Thm}[Ash--Stevens, Urban, Hansen]\label{overconv}
Let $\tilde{t}=\prod_{i\in\Delta}t_i.$
\begin{enumerate}[(i)]
\item\label{overconv1} The space $\HH^{d}(X_{K^\p\times I_\p},\Di^{n}_{\cf})^{\epsilon}$ admits a slope decomposition with respect to $U_{\tilde{t}}$ and every rational number $h$. 
\item\label{overconv2} If $h$ is small enough with respect to $\tilde{t}$ and the trivial character $\cf$ (see for example the first formula on page 1690 of \cite{Ur} or equation $(21)$ of \cite{AshSt}), then the map
$$
\HH^{d}(X_{K^\p\times I_\p},\Di^{n}_{\cf})^{\epsilon,\leq h}  \too \HH^{d}(X_{K^\p\times I_\p},\Co)^{\epsilon,\leq h}.
$$
is an isomorphism.
In particular, this is an isomorphism for $h=0$.
\end{enumerate}
\end{Thm}
\begin{proof}
For the first claim see Section 2.3 of \cite{Hansen} and for the second one see \cite{Hansen}, Theorem 3.2.5.
The third claim is \cite{Ur}, Proposition 4.3.10.
Note that in all cases the authors consider all primes lying above $p$ at once.
The same proofs work in our partial $\p$-adic setup.
\end{proof}

Composing the dual of the augmentation map
\begin{align}
\aug_\cf^{\ast}\colon \HH^{d}(G(F),\Fu^{\cont}_\Co(K^\p,\II^{\an}_{B}(\Co);\Co(\epsilon)))\too \HH^{d}(X_{K^\p\times I_\p},\Di^{n}_{\cf})^{\epsilon}
\end{align}
with the map \eqref{classical} yields the map
\begin{align}\label{principallift}
\HH^{d}(G(F),\Fu^{\cont}_\Co(K^\p,\II^{\an}_{B}(\Co);\Co(\epsilon)))\too \HH^{d}(X_{K^\p\times I_\p},\Co)^{\epsilon}.
\end{align}

As the diagram
\begin{center}
\begin{tikzpicture}
    \path 	
	  (0,1.5) 	node[name=A]{$\cind_{I_\p}^{G_\p}\Co$}
		(0,0) 	node[name=C]{$\St_{G_\p}^{\infty}(\Co)$}
		(6,1.5) 	node[name=B]{$\II^{\an}_{B}(\Co)$}
		(6,0) 	node[name=D]{$\St_{G_\p}^{\cont}(\Co)$}
		(3,1.5) node[name=E]{$\cind_{I_\p}^{G_\p} \An^{n}_{\cf}$};
    \draw[->] (A) -- (E) node[midway, above]{\eqref{Iwahoriinclusion}};
		\draw[->] (E) -- (B) node[midway, above]{$\aug_\cf$};
    \draw[->] (A) -- (C) node[midway, left]{\eqref{evaluationzero}};
		\draw[->] (C) -- (D) ;
		\draw[->] (B) -- (D) ;
\end{tikzpicture} 
\end{center}
is commutative up to multiplication with a non-zero constant, which comes from the choice of a Iwahori-fixed vector in \eqref{evaluationzero}, so is the induced diagram in cohomology: 
\begin{center}
\begin{tikzpicture}
    \path 	
	  (0,1.5) 	node[name=A]{$\HH^{d}(G(F),\Fu^{\cont}_\Co(K^\p,\II^{\an}_{B}(\Co);\Co(\epsilon)))$}
		(0,0) 	node[name=C]{$\HH^{d}(G(F),\Fu^{\cont}_\Co(K^\p,\St_{G_\p}^{\cont}(\Co);\Co(\epsilon)))$}
		(6.5,1.5) 	node[name=B]{$\HH^{d}(X_{K^\p\times I_\p},\Co)^{\epsilon}$}
		(6.5,0) 	node[name=D]{$\HH^{d}(G(F),\Fu^{\cont}_\Co(K^\p,\St_{G_\p}^{\infty}(\Co);\Co(\epsilon)))$};
    \draw[->] (A) -- (B) node[midway, above]{\eqref{principallift}};
    \draw[->] (C) -- (A) ;
		\draw[->] (C) -- (D) node[midway, above]{$\cong$};;
		\draw[->] (D) -- (B) node[midway, right]{$\ev^{(d)}$};;
\end{tikzpicture} 
\end{center}

\begin{Pro}\label{principalisom}
Let $a_i\in E^\times$, $i\in\Delta$, be elements with $p$-adic valuation equal to $0$ and $\chi_{\underline{a}}\colon T\to E^\times$ the unique smooth unramified character such that $\chi_{\underline{a}}(t_i)=a_i.$
The natural inclusion $i^{\infty}_{B}(\chi_{\underline{a}})\into\II^{\an}_{B}(\chi_{\underline{a}})$ induces an isomorphism
$$\HH^{d}(G(F),\Fu^{\cont}_\Co(K^\p,\II^{\an}_{B}(\chi_{\underline{a}});\Co(\epsilon)))\xrightarrow{\cong} \HH^{d}(G(F),\Fu_\Co(K^\p,i^{\infty}_{B}(\chi_{\underline{a}});\Co(\epsilon)))$$
for all $d\geq 0$.
In particular, the map \eqref{principallift} induces an isomorphism
$$\HH^{q}(G(F),\Fu^{\cont}_\Co(K^\p,\II^{\an}_{B}(\Co);\Co(\epsilon)))[\pi^\p]\xrightarrow{\cong} \HH^{q}(X_{K^\p\times I_\p},\Co)^{\epsilon}[\pi]$$
on $\pi$-isotypic components in degree $q$.
\end{Pro}
\begin{proof}
The second claim is a direct consequence of the first claim and Lemma \ref{injectivity}.

Given an $I_\p$-representation $M$ we denote by $\Coind_{I_\p}^{G_\p}M$ the coinduction of $M$ to $G_\p$, i.e.~the module of all functions $f\colon G_\p\to M$ such that $f(kg)=kf(g)$ for all $k\in I_\p$, $g\in G_\p$.
By taking continuous duals the Koszul complex of Theorem \ref{Koszulres} with $y_i=U_{t_i}-a_i$ yields a quasi-isomorphism
$$\Hom_E^{\cont}(\II^{\an}_{B}(\chi_{\underline{a}}),E) \xrightarrow{\cong} \Lambda^{\bullet}_\Co(\Co^{\Delta}) \otimes_\Co \Coind_{I_\p}^{G_\p}(\Di^{n}_{\cf})$$
of complexes.
Note that taking continuous duals in this case is exact by the Hahn--Banach theorem.
Similarly by the smooth Koszul resolution \eqref{Koszulsmooth} we have a quasi-isomorphism
$$\Hom_E(i^{\infty}_{B}(\chi_{\underline{a}}),E) \xrightarrow{\cong} \Lambda^{\bullet}_\Co(\Co^{\Delta}) \otimes_\Co \Coind_{I_\p}^{G_\p}(\Co)$$
and the canonical diagram
\begin{center}
\begin{tikzpicture}
    \path 	
	  (0,1.5) 	node[name=A]{$\Hom_E^{\cont}(\II^{\an}_{B}(\chi_{\underline{a}}),E)$}
		(4,1.5) 	node[name=C]{$\Lambda^{\bullet}_\Co(\Co^{\Delta}) \otimes_\Co \Coind_{I_\p}^{G_\p}(\Di^{n}_{\cf})$}
		(0,0) 	node[name=B]{$\Hom_E(i^{\infty}_{B}(\chi_{\underline{a}}),E)$}
		(4,0) 	node[name=D]{$\Lambda^{\bullet}_\Co(\Co^{\Delta}) \otimes_\Co \Coind_{I_\p}^{G_\p}(\Co)$};
    \draw[->] (A) -- (B) ;
    \draw[->] (A) -- (C) ;
		\draw[->] (C) -- (D) ;
		\draw[->] (B) -- (D) ;
\end{tikzpicture} 
\end{center}
is commutative.

From the associated spectral sequences for double complex we deduce that it is enough to prove that the map of complexes
\begin{align}\label{qisom}
\Lambda^{\bullet}_\Co(\Co^{\Delta})\otimes_\Co \HH^{d}(X_{K^\p\times I_\p},\Di^{n}_{\cf}) \too \Lambda^{\bullet}_\Co(\Co^{\Delta})\otimes_\Co \HH^{d}(X_{K^\p\times I_\p},\Co)
\end{align}
is a quasi-isomorphism for all $d\geq 0.$
The modules on the left hand side admit a slope decomposition for $U_{\tilde{t}}$ by Theorem \ref{overconv} \eqref{overconv1}, while the modules on the right hand side admit a slope decomposition since they are finite-dimensional $\Co$-vector spaces.

Since the operators $y_i$ commute with $U_{\tilde{t}}$ they respect the slope decomposition on both sides (see \cite{Ur}, Lemma 2.3.2).
On the slope less or equal than $0$ part the map $\eqref{qisom}$ is even an isomorphism of complexes by Theorem \ref{overconv} \eqref{overconv2}.

Thus, we are reduced to prove that
\begin{align*}
\Lambda^{\bullet}_\Co(\Co^{\Delta})\otimes_\Co \HH^{d}(X_{K^\p\times I_\p},\Di^{n}_{\cf})^{>h} \too \Lambda^{\bullet}_\Co(\Co^{\Delta})\otimes_\Co \HH^{d}(X_{K^\p\times I_\p},\Co)^{>h}
\end{align*}
is a quasi-isomorphism for all $d\geq 0.$.

In fact, both sides are acyclic:
standard properties of Koszul complexes imply that the operators $y_i$ act as multiplication by zero on the cohomology of the complexes.
By the definition of a slope decomposition the operator $U_{\tilde{t}}-\prod_{i\in\Delta} a_i$ acts via isomorphisms on both complexes and thus on their cohomology.
But since $U_{\tilde{t}}-\prod_{i\in\Delta} a_i$ lies in the ideal generated by the operators $y_i$ it also acts via multiplication by zero, which proves the claim.
\end{proof}

\begin{Rem}
In order to keep the article short and avoid unnecessary notation we decided to stick to the special case above.
But Proposition \ref{principalisom} also holds in a much more general situation, e.g.~one can allow arbitrary coefficient systems and arbitrary non-critical principal series representations, given that the slope is small enough w.r.t. $\tilde{t}$.
\end{Rem}

%% file: Big-OC-Families.tex
\subsection{Overconvergent families}\label{Families}

Let $\Ws$ be the weight space of $T$, i.e.~it is the rigid space over $\Co$ such that for every $\Co$-affinoid algebra $A$ its $A$-points are given by
$$
\Ws(A)=\Hom^{\cont}(T,A^{\ast}).
$$
It is smaller then the usual weight space as we only consider characters of the torus in $G_{F_\p}$.
For any open affinoid $\mathcal{U} \subset \Ws$ and we let $\chi_{\mathcal{U}}$ be the corresponding universal weight.
As in Section \ref{Koszul} we define the space of $n$-analytic functions and distributions $\An_{\chi_{\mathcal{U}}}^{n}$ and $\Di^n_{\chi_{\mathcal{U}}}$ for $n\gg 0$.

We suppose now that the group $G_\infty$ has discrete series or, equivalently, that $\delta=0$ (see \cite{Knapp}, Theorem 12.20 for the equivalence of these two properties.).
Thus by Proposition \ref{componentpro}, the representation $\pi^{\infty}$ appears only in the middle degree cohomology $\HH^q(X_{K^{\p}\times I_{\p}}, \Co)$.
Hence, we put $\LI_{i}(\pi,\p)^{\epsilon}=\LI_{i}^{(q)}(\pi,\p)^{\epsilon}$.

Let $\T_{\sph}^{\p}\subseteq \T^{\p}$ be the spherical Hecke algebra of level $K^{\p}\times I_{\p}$ over $\Co$, i.e.~the commutative Hecke algebra generated by all Hecke operators at finite places $v$ such that $K_v$ is hyperspecial.
We define $\T_{\sph}$ to be the commutative algebra generated by $\T_{\sph}$ and all $U_t$-operators for $t\in T^{-}.$
Let $\m_\pi\subseteq \T_{\sph}$ (resp.~$\m_\pi^{\p}\subseteq \T_{\sph}^{\p}$) the maximal ideal associated with $\pi.$
We assume the following weak non-Eisenstein assumption on the maximal ideal $\m_\pi$ throughout this section.
\begin{Hyp}[NE]\label{Hyp2}
We have
$$\HH^d(X_{K^{\p}\times I_{\p}}, \Co)_{\m_\pi}=0$$
unless $d=q$.
\end{Hyp}
\begin{Exa}
If the group $G$ is definite, the hypothesis is automatically true.
By strong multiplicity one and the Jacquet--Langlands correspondence it is also true for inner forms of $\PGL_2$ over totally real number fields.
\end{Exa}

Let $\mathcal{U}$ be an open affinoid of $\Ws$ and define $\mathcal{O}_{\mathcal{U},\cf}$ to be the rigid localisation of $\mathcal{O}_{\Ws}(\mathcal{U}) $ at the weight $\cf$ (which is the cohomological weight of $\pi$), i.e.~
$$
\mathcal{O}_{\mathcal{U},\cf}= \varinjlim_{\cf  \in  \mathcal{U}' \subset \mathcal{U}} \mathcal{O}_{\Ws}(\mathcal{U}').
$$
The following theorem is instrumental in calculating $\LI$-invariants.

\begin{Thm}\label{thm:locfree}
After localisation at the ideal $\m_{\pi}$ of $\T_{\sph}$ and restricting to a small enough open affinoid $\mathcal{U}$ containing $\cf$ the canonical reduction map
$$
\HH^q (X_{K^{\p}\times I_{\p}}, \Di^{n}_{\chi_{\mathcal{U}}} )^{\epsilon}_{\m_{\pi}}\rightarrow \HH^q(X_{K^{\p}\times I_{\p}}, \Co)^{\epsilon}_{\m_{\pi}}
$$
is surjective.
Moreover, $\HH^q (X_{K^{\p}\times I_{\p}}, \Di^{n}_{\chi_{\mathcal{U}}} )^{\epsilon}_{\m_\pi}$ is a free $\mathcal{O}_{\mathcal{U},\cf}$-module of rank equal to the dimension of $\HH^q(X_{K^{\p}\times I_{\p}}, \Co)^{\epsilon}_{\m_{\pi}}$.

In addition, we have
$\HH^d (X_{K^{\p}\times I_{\p}}, \Di^{n}_{\chi_{\mathcal{U}}} )^{\epsilon}_{\m_{\pi}}=0$
for all $d \neq q$.
\end{Thm}

\begin{proof}
The case of $\PGL_2$ over a totally real field is proven in detail in \cite{BDJ}, Theorem 2.14.
The main ingredient in their proof is the vanishing of the cohomology outside middle degree.
Thus, the same proof works in our more general setup.
\end{proof}

\begin{Def}
We say that $\m_\pi$ is $\p$-étale (with respect to $\epsilon$) if every Hecke operator $h\in \T_{\sph}$ acts on $\HH^q (X_{K^{\p}\times I_{\p}}, \Di^{n}_{\chi_{\mathcal{U}}} )^{\epsilon}_{\m_{\pi}}$ as multiplication by an element $\alpha_h\in \mathcal{O}_{\mathcal{U},\cf}^{\ast}.$
\end{Def}

Theorem \ref{thm:locfree} immediately implies the following.
\begin{Cor}
Assume that $\dim_\Co \HH^q(X_{K^{\p}\times I_{\p}}, \Co)^{\epsilon}_{\m_{\pi}}=1$.
Then $\m_\pi$ is $\p$-étale.
\end{Cor}

\begin{Exa}
Suppose $G$ is an inner form of $\PGL_2$ over a totally real number field.
Then the above corollary together with strong multiplicity one implies that $\m_\pi$ is $\p$-étale.
\end{Exa}

Let $t_i\in T^{-}$, $i\in \Delta$ be a choice of elements as in Section \ref{Koszul}.
Suppose $\m_\pi$ is $\p$-étale (with respect to $\epsilon$) with associated eigenvalues $\alpha_{U_{t_i}}\in \mathcal{O}_{\mathcal{U},\cf}^{\ast}$, $i \in \Delta$.
By possibly shrinking $\mathcal{U}$ we may assume that $\alpha_{U_{t_i}}\in \mathcal{O}_{\mathcal{U}}^{\ast}.$
Every element $t\in T$ can be written uniquely as a product $t=t_0 \prod_{i\in\Delta} t_i^{n_i}$ with $t_0 \in T_0$, $n_i\in \Z$.
Hence, there exists a unique character $\chi_\alpha\colon T\to \mathcal{O}_{\mathcal{U}}^{\ast}$ such that
$$\left.\chi_{\alpha}\right|_{T_0}=\chi_{U}$$
and
$$\chi_\alpha(t_i)=\alpha_{U_{t_i}}.$$

\begin{Thm}\label{thm:Linvderivative}
Suppose that $\m_\pi$ is $\p$-étale (with respect to $\epsilon$) with associated character $\chi_\alpha\colon T \to \mathcal{O}_{\mathcal{U}}^{\ast}$.
Then for every tangent vector $v$ of $\mathcal{U}$ at $\cf$ we have:
$$
\frac{\partial}{\partial v}\chi_{\alpha,i}\in \LI_{i}(\pi,\p)^{\epsilon}
$$
where $\chi_{\alpha,i}$ denotes the composition $\chi_{\alpha}\circ i^\vee.$
Moreover, the subspace $\LI_{i}(\pi,\p)^{\epsilon}\subseteq \Hom^{\cont}(F_\p^{\ast},\Co)$ has codimension one.
\end{Thm}
\begin{proof}
Let $\m_\cf\subseteq \mathcal{O}_{\mathcal{U}}$ be the maximal ideal corresponding to the trivial character.
We put $A=\mathcal{O}_{\mathcal{U}}/\m_{\cf}^{2}$ and $\chi_0 = \chi_{\mathcal{U}} \bmod \m_{\cf}^{2}$.
Let us write $\widetilde{\chi}_\alpha= \chi_\alpha \bmod \m_{\cf}^{2}\colon T \to A.$
Using Theorem \ref{thm:locfree} and arguments with Koszul complexes as in the proof of Proposition \ref{principalisom} one shows that the image of the map
$$\HH^{q}(G(F),\Fu^{\cont}_{A}(K^\p,\II_{B}^{\an}(\widetilde{\chi}_\alpha);A(\epsilon)))_{\m_{\pi}^{\p}}\too \HH^{q} (X_{K^{\p}\times I_{\p}}, E )^{\epsilon}_{\m_{\pi}^{\p}}$$
is the intersection of the kernels of the homomorphisms $U_{t_i}-1$.
Therefore, the first claim follows from Proposition \ref{CGS}.

We may assume that $\Co$ is large enough.
We explain how to choose appropriate tangent vectors so that $\LI_{i}(\pi,\p)^{\epsilon}$ contains elements of the form $\log_{p,\sigma}-\LI^{\sigma}\ord_\p$ for every embedding $\sigma\colon F_\p \to \Co$. This will show that  the $\LI$-invariant has codimension at most one.

Let $\Co \langle k_\sigma \rangle$ be the ring of power series in the $k_\sigma$'s and let $\mathcal{O}_{\p,1}^\ast$ denote the free part of $\mathcal{O}_{\p}^\ast$. By smoothness of the weight space, we can embed $ \mathcal{O}_{\mathcal{U}}^{\ast}$ in $\Co \langle k_\sigma \rangle$ and the structural morphism of  $\Co \llbracket \mathcal{O}_{\p,1}^\ast \rrbracket$ into   $\mathcal{O}_{\mathcal{U}}^{\ast}$ can be described via the map 
\[
\mathcal{O}_\p^\ast \ni <u> \mapsto \prod_{\sigma } \sigma(u)^{k_{\sigma}}.
\]
We can hence identify $\chi_{\alpha,i}$ with a map from $F_{\p}^{\ast}$ to $\mathcal{O}_{\mathcal{U}}^{\ast}$.
For all $n \in \mathbb{Z}$ and $u \in \mathcal{O}_{\p,1}^\ast$ we have
\[
\chi_{\alpha,i}(\pi^n u)=\alpha_{U_{t_i}}^n \prod_{\sigma } \sigma(u)^{k_{\sigma}}.
\]
Note that 
\[
\frac{\textup{d}}{ \textup{d}k_{\sigma'}} \left(\sigma(u)^{k_{\sigma}} \right)=\delta_{\sigma,\sigma'} \log_p(\sigma(u))\sigma(u)^{k_{\sigma}},
\]
where
\[
\delta_{\sigma,\sigma'} = \begin{cases} 1 & \mbox{if}\ \sigma=\sigma' \\ 0 & \mbox{else.}\end{cases}
\]
Take for $v$ the direction where only $k_{\sigma}$ varies, derive $\chi_{\alpha,i}$ along $v$, and evaluate at $k_{\sigma'}=0$ for all $\sigma'$ (corresponding to the weight $\cf$) to get 
\[
\frac{\partial}{\partial v}\chi_{\alpha,i}(\pi^n u)= \frac{\textup{d}}{ \textup{d}k_{\sigma}}\alpha_{U_{t_i}}(\cf)\ord_\p(\pi^n)+\alpha_{U_{t_i}}(\cf)\log_p(\sigma(u)).
\]
By the Steinberg hypothesis, $\alpha_{U_{t_i}}(\cf)$ is not vanishing and we are done.

By Proposition \ref{codimension} the $\LI$-invariant has codimension at least one and, hence, the second claim follows.
\end{proof}

%% file: Big-App-JL.tex
\subsection{Hilbert modular forms}\label{JL}
We want to study the case of inner forms of $\PGL_2$, which are split at $\p$, over a totally real number field $F$ in detail.
Since there is only one simple root we drop it from the notation.
If $G$ is equal to $\PGL_2$, one can attach $2^{[F:\Q]}$ a priori different $\LI$-invariants $\LI(\pi,p)^{\epsilon}$ to $\pi$; in this case, a conjecture of Spie\ss~(cf.~\cite{Sp}, Conjecture 6.4) states that the definition doesn't depend on $\epsilon$, i.e~
$$\LI(\pi,\p)^{\epsilon}=\LI(\pi,\p)^{\epsilon'}$$
for all choices of sign characters $\epsilon$ and $\epsilon'$. 
In the same paper, Remark 6.6(b), Spie\ss{} also states that ``an interesting and difficult problem'' is to show that the $\LI$-invariant is invariant under Jacquet--Langlands transfers. 
In this section, thanks to Theorem \ref{thm:Linvderivative}, we will settle both, Spie\ss' conjecture and his question.

Moreover, let $\rho_{\pi}\colon\Gal(\overline{\Q}/F)\to \GL_2(\Co)$ be the Galois representation associated with $\pi$ (or rather associated with the Jacquet--Langlands transfer $\JL(\pi)$ of $\pi$ to $\PGL_2$), which exists by work of Taylor (see \cite{TaylorHMF}).
As $\pi$ is $\p$-ordinary, by Theorem 2 of \cite{WilesLambda} we know that the restriction $\rho_{\pi,\p}\colon \Gal(\overline{\Q_p}/F_\p)$ of $\rho_{\pi}$ to a decomposition group at $\p$ is ordinary. Moreover, since $\pi_\p$ is Steinberg a result of Saito (see \cite{Saito}) implies that this restriction  is a non-split extension of the trivial character by the cyclotomic character.
Therefore, it gives a class $\langle\rho_{\pi,\p}\rangle \in \HH^{1}(\Gal(\overline{\Q_p}/F_\p),\Co(1))$, where $\Co(1)$ denotes the Tate twist of $\Co$.
The Fontaine--Mazur $\LI$-invariant $\LI^{\FM}(\rho_{\pi,\p})$ of $\rho_{\pi,\p}$ is the orthogonal complement of $\langle\rho_{\pi,\p}\rangle$ with respect to the local Tate pairing
$$\HH^{1}(\Gal(\overline{\Q_p}/F_\p),\Co(1)) \times \HH^{1}(\Gal(\overline{\Q_p}/F_\p),\Co) \too \Co.$$
By local class field theory we have a canonical isomorphism
$$\HH^{1}(\Gal(\overline{\Q_p}/F_\p),\Co)\cong \Hom^{\cont}(F_\p^{\ast},\Co)$$
and, thus, we consider $\LI^{\FM}(\rho_{\pi,\p})$ as a codimension one subspace of $\Hom^{\cont}(F_\p^{\ast},\Co)$.
We show that automorphic $\LI$-invariants equal the Fontaine--Mazur $\LI$-invariant of the associated Galois representation.

\begin{Thm}\label{thm:HMF}
Suppose $G$ is an inner form of $\PGL_2$ over a totally real number field $F$, which is split at the prime $\p$ of $F$.
Let $\pi$ be a cuspidal automorphic representation of parallel weight $2$ of $G$ that is Steinberg at $\p$.
\begin{enumerate}[(i)]
\item The automorphic $\LI$-invariant of $\pi$ is independent of $\epsilon$, i.e.~
$$\LI(\pi,\p)^{\epsilon} = \LI(\pi,\p)^{\epsilon'}$$
for all choices of $\epsilon$ and $\epsilon'$. 
We therefore put $\LI(\pi,\p)=\LI(\pi,\p)^{\epsilon}$ for any choice of sign character $\epsilon$.
\item Let $\JL(\pi)$ be the Jacquet--Langlands transfer of $\pi$ to $\PGL_2$.
The equality
$$\LI(\pi,\p) = \LI(\JL(\pi),\p)$$
holds.
\item Let $\rho_\pi$ the Galois representation associated with $\pi$.
Then:
$$\LI(\pi,\p)=\LI^{\FM}(\rho_{\pi,\p}).$$
\end{enumerate}
\end{Thm}
\begin{proof}
We actually prove that $\LI(\pi,\p)^{\epsilon}=\LI^{\FM}(\rho_{\pi,\p})$ for every $\epsilon$, which implies all other claims.
Since $\m_\pi$ is $\p$-étale (with respect to $\epsilon$) we can deform the Hecke eigenvalues of $\pi$ to a family over an open affinoid subspace of weight space in the following sense:
there exists an open affinoid $\mathcal{U}\subseteq \Ws$ containing the trivial character such that the eigenvalue $\alpha_h$ of each Hecke operator $h\in \T_{\sph}$ acting on $\HH^q (X_{K^{\p}\times I_{\p}}, \Di^{n}_{\chi_{\mathcal{U}}} )^{\epsilon}$ is an element of $\mathcal{O}_{\Ws}(\mathcal{U})$.
As explained in \cite{BDJ}, Theorem 2.14 (ii), we may shrink $\mathcal{U}$ further such that the common eigenspace associated to the eigenvalues $\alpha_h$ is a free $\mathcal{O}_{\Ws}(\mathcal{U})$-submodule of $\HH^q (X_{K^{\p}\times I_{\p}}, \Di^{n}_{\chi_{\mathcal{U}}} )^{\epsilon}$ of rank equal to the dimension of $\HH^q(X_{K^{\p}\times I_{\p}}, \Co)^{\epsilon}_{\m_{\pi}}$.
By shrinking $\mathcal{U}$ even further we may assume that $\alpha_{t_i}\in \mathcal{O}_{\Ws}(\mathcal{U})^\times$ has absolute value one for each $i\in\Delta$. 
then, by the classicality of small slope overconvergent eigenforms the specialization of these eigenvalues at a classical point $\lambda\in\mathcal{U}$ are the eigenvalues associated to a cuspidal representation $\pi_\lambda$ of weight $\lambda$.

Note that the Galois representation $\rho_{\pi}$ is irreducible, as the Hilbert modular form associated with $\pi$ is cuspidal. Thus, by standard arguments (see for example Theorem B of \cite{ChenevierDet}) there exists (after possibly shrinking $\mathcal{U}$ again) a family of Galois representations $\rho_{\pi,\mathcal{U}}$  over $\mathcal{U}$ passing through $\rho_{\pi}$, i.e.
the specialization of $\rho_{\pi,\mathcal{U}}$ at each classical point $\lambda\in\mathcal{U}$ is isomorphic to the Galois representation attached to $\pi_\lambda.$
It follows from the results of Saito (see Theorem 1 of \cite{Saito}) that for every classical point $z\in \mathcal{U}$ the restriction of the Galois representation $\rho_{\pi,z}$ to a decomposition group is determined by the $U_\p$-eigenvalue. Moreover, the Weil--Deligne module of $\rho_{\pi}$ restricted at the decomposition group at $\p$ is the Weil--Deligne module associated with the Steinberg representation via the local Langlands correspondence. In particular, our representation (with the filtration induced by the ordinarity of the Galois representation) is non-critical special, as defined in \S 3.1 of  \cite{Ding}; indeed we know that the corresponding $(\varphi,\Gamma)$-module is an extension of special characters, so we just have to check if the extensions are split or not. If one extension were split, then  the Weil--Deligne representation would also be split, contradicting the fact the monodromy for the Steinberg is maximal. 

Comparing the Galois theoretic Colmez--Greenberg--Stevens formula (see \cite{Ding}, Theorem 3.4) with its automorphic counterpart of Theorem \ref{thm:Linvderivative} yields the result.

\end{proof}

\begin{Rem}
\begin{enumerate}[(i)]
\item In the Hilbert modular form case the constructions simplify substantially and, thus, one can check that our methods also work for higher weights assuming that the representation is non-critical.
\item
One can prove the first two claims of Theorem \ref{thm:HMF} without passing to the Galois side.
Namely, using methods as in Section 5 of \cite{Hansen} one can show that the eigenvalue of the overconvergent $\p$-adic family passing through $\pi$ is independent of the sign character and stable under the Jacquet--Langlands transfer.
So by Theorem \ref{thm:Linvderivative} we can conclude the argument.
\item By definition Fontaine--Mazur $\LI$-invariants are stable under restricting the Galois representation to the absolute Galois group of finite extensions.
Thus, we can deduce from Theorem \ref{thm:HMF} that automorphic $\LI$-invariants of Hilbert modular forms are stable under abelian base change with respect to totally real extensions.
\item In the case of modular elliptic curves over totally real fields of class number one the equality of automorphic and Fontaine--Mazur $\LI$-invariants was conjectured by Greenberg (see \cite{Greenberg}, Conjecture 2).
Theorem \ref{thm:HMF} thus implies that the construction of Stark--Heegner points over totally real fields is unconditional (see \cite{GMS} for a detailed discussion of Stark-Heegner points).
\end{enumerate}
\end{Rem}

%% file: Big-App-Galois.tex
\subsection{Unitary groups}\label{Galois}
Let $\tilde{F}$ be a CM field with totally real subfield $F$.
We assume that the prime $\p$ of $F$ is split in $\tilde{F}$.
Let $U$ be the unitary group attached to a positive definite hermitian space over $\tilde{F}$ and $G$ the associated adjoint group.
By construction $G_\p$ is isomorphic to $\PGL_n(F_\p)$.
We can identify the simple roots with respect to the upper triangular Borel with the set $\left\{1,\ldots, n-1\right\}$.
Let $\pi$ be an automorphic representation of $G$ such that $\pi_\infty=\C$ and $\pi_\p$ is the Steinberg representation of $\PGL_n(F_\p)$.
In this case the only sign character is the trivial character and therefore we drop it from the notation.

By Shin's appendix to \cite{Goldring} the base change $\BC(\pi)$ of $\pi$ to $\PGL_n$ over $\tilde{F}$ exists and it is Steinberg at both primes of $\tilde{F}$ lying above $\p$.
By the work of many people (see Theorem 2.1.1 of \cite{BLGGT} for a detailed discussion) we can attach a $p$-adic Galois representation
$$\rho_{\pi}=\rho_{\BC(\pi)}\colon \Gal(\overline{\Q}/\tilde{F})\too \GL_n(\Co)$$
to $\BC(\pi)$.
Since we consider the trivial coefficient system the Steinberg representation is ordinary (cf.~\cite{Geraghty}, Lemma 5.6).
Therefore, as shown in \cite{Thorne}, Theorem 2.4, one deduces from the local-global compatibility theorem of Caraiani (see \cite{Caraiani}) that the restriction $\rho_{\pi,\p}$ of $\rho_{\pi}$ to the decomposition group of a prime above $\p$ can be brought in the following form:
it is upper triangular and the $i$-th diagonal entry is the $1-i$-power of the cyclotomic character.

Therefore, we have $n-1$ canonical $2$-dimensional subquotients $\rho_{\pi,\p,i}$ which are extensions of $\Co(-n+i)$ by $\Co(-n+i+1)$.
We can consider the associated Fontaine--Mazur $\LI$-invariant $\LI^{\FM}_{i}(\rho_{\pi,\p})=\LI^{\FM}(\rho_{\pi,\pi,i}(n-i))$. Recall that the Weil--Deligne representation associated with the Steinberg has always maximal monodromy so the corresponding $(\varphi,\Gamma)$-module is non-critical split \`a la Ding, see \S 3.1 of \cite{Ding}. Replacing the local-global compatibility results of Saito by those of Caraiani (see \cite{Caraiani}), we can apply Theorem 3.4 of \cite{Ding} to the $(\varphi,\Gamma)$-module associated with the  Galois representation $\rho_{\pi}$ and then the same proof as that of Theorem \ref{thm:HMF} yields the following result.
\begin{Thm}\label{thm:unitary}
Let $F$ be a totally real number field and $G$ the adjoint of a unitary group over $F$ compact at infinity and split at $\p$.
Let $\pi$ be an automorphic representation of $G$ such that $\pi_\infty=\C$ and $\pi_{\p}$ is Steinberg.
Suppose $\pi$ satisfies (SMO), that we can choose the tame level $K^{\p}$ such that $\m_{\pi}$ is $\p$-étale and that the Galois representation $\rho_{\pi}$ attached to $\pi$ is irreducible.
Then we have
$$\LI^{\FM}_{i}(\rho_{\pi}) = \LI_{i}(\pi,\p)$$
for every $i=1,\ldots,n-1$.
\end{Thm}

\begin{Rem}
Something can be said also when $F =\mathbb{Q}$ and $G=\Sp_{2g}$. 
In this case the Galois representations have been constructed by Scholze \cite{Scholze} but local-global compatibility is not known.
Still, if one suppose that the $2g+1$-dimensional Galois representation $\Std(\rho_{\pi})$ associated with $\pi$ is semistable with maximal monodromy then the Greenberg--Benois $\LI(\Std(\rho_{\pi}))$-invariant has been calculated in \cite[Theorem 1.3]{RossoLinv}. 
 
In this case the root system of $\Sp_{2g}$ can be identified with the set $\left\{1,\ldots,g\right\}$ via the identification with the root system of $\GL_g$, which is  embedded in $\Sp_{2g}$ by
$$\GL_g\too \Sp_{2g},\ A\mapstoo \left( \begin{array}{cc}
\phantom{e}^t A^{-1} & 0\\
0 & A
\end{array}  \right),$$
and we see that the automorphic $\LI$-invariant $\LI_{1}(\pi,p)$ coincides with the one of \cite{RossoLinv}.

The same calculation in Section 4.2 of {\it loc.~cit.~}gives us that $\LI_{i}(\pi,p)$ is the Greenberg--Benois $\LI$-invariant for $\Std(\rho_{\pi})(i-1)$.
(This case has not been treated in {\it loc.~cit.~}as the $L$-values are not Deligne-critical but Benois' definition applies also in this case, see formula (96) in \cite{BenoisMemoire}.)

If $F_\p \neq \Q_p$, then the comparison is more subtle, as there is only just one Greenberg--Benois $\LI$-invariant per $p$-adic place, and from the Galois side one needs to consider Galois invariant characters of $F^*_\p$. 
\end{Rem}

%% file: Big-Bianchi.tex
When $G_\infty$ does not fulfil the Harish-Chandra condition, the representation $\pi$ contributes to several degrees of cohomology of the associated locally symmetric space and the techniques used in Theorem \ref{thm:locfree} break down. 

There are two tools to tackle the problem:
first, one can use Hansen's Tor-spectral sequence (see Theorem 3.3.1 of \cite{Hansen})
$$
\Tor_{i}^{\mathcal{O}_{\Ws}} (\HH^{q+i}(X_{K^{\p}\times I_{\p}}, \Di_{\chi_{\mathcal{U}}})^{\leq h}, \Co ) \Longrightarrow \HH^q(X_{K^{\p}\times I_{\p}}, \Di_{\chi_{\Sigma}})^{\leq h},
$$
where $\mathcal{U}\subseteq \Ws$ is an open affinoid and $\Sigma\subseteq\mathcal{U}$ is Zariski-closed, to analyse the overconvergent cohomology groups in question;
second, one can use cases of Langlands functoriality in $p$-adic families to reduce to groups, which fulfil the Harish-Chandra condition.

In good situations it should be possible to calculate at least one of the $\LI$-invariants $\LI^{(d)}_{i}(\pi,\p)$ for $0 \leq  d\leq \delta$, using Proposition \ref{CGS}.
The main difficulty is that in general most classes do not lift to a big cohomology class as classes in $\HH^{q+1}$ give lifting obstructions.
But as soon as one can show that at least a class lifts to a family in $\HH^{q+d}$ (combined with some results on the tangent directions in the eigenvariety) Proposition \ref{CGS} lets us calculate $\LI^{(d)}_{i}(\pi,\p)$.
If Venkatesh's conjecture (stating that the $\pi$-isotypic component of the cohomology is generated by the minimal degree cohomology as a module over the derived Hecke algebra) holds than these $\LI$-invariants $\LI^{(d)}_{i}(\pi,\p)$, for varying $d$, are essentially all the same by the main result of \cite{Ge3}.

Using unpublished work of Hansen (see also \cite{BWi2}) we shall study the case of Bianchi modular forms.
We now fix $F$ to be a quadratic imaginary field where $p$ is unramified and $\pi$ a cuspidal representation of $\PGL_{2,F}$ of parallel weight $2$ such that
$\pi_\q$ is Steinberg for all primes $\q$ lying above $p$.
We put $G=\PGL_{2,F}$ if $p$ is inert and, if $p$ is split, we define $G$ to be the Weil restriction $\Res_{F/\Q}\PGL_{2,F}$.
In the first case there is only one simple root and thus we drop it from the notation.
In the second case, the simple roots can be identified with the two primes above $\p$.
In both cases, the only sign character is the trivial one and we shall drop it from the notation as well.
(The theorem below indicates that in the case of a split prime the partial eigenvarieties we considered before may not be big enough.
The case that $p$ is split and $\pi$ is only Steinberg at one of the primes above $p$ could be handled similarly but one would have to introduce new notations.
For the sake of brevity, we do not discuss it further.)
  
We recall some result on the eigenvariety for Bianchi modular forms due to Hansen and Barrera--Williams (see \cite{BWi2}, Lemma 4.4 and the proof of Theorem 4.5).  
\begin{Thm}(Hansen, Barrera--Williams)\label{thm:HBSW}
Let $\mathcal{U}\subseteq\Ws$ be an open affinoid neighbourhood of the trivial character.
\begin{enumerate}[(i)]
\item\label{HBSW1} The system of eigenvalues associated with $\pi$ appears in $\HH^d(X_{K^{\p}\times I_{\p}}, \Di_{\chi_{\mathcal{U}}})$ if and only if $d=2$.
\item\label{HBSW2} There is at least one curve $\mathcal{S} \subseteq \mathcal{U}$ passing through $\cf$ such that 
$$\HH^1(X_{K^{\p}\times I_{\p}}, \Di_{\chi_{\mathcal{S}}})_{\m_{\pi}}\neq 0.$$
If such a curve $\mathcal{S}$ is smooth at $\cf$, the space is free of rank one over $\mathcal{O}_{\mathcal{S}, \cf}$ and the canonical map
$$\HH^1(X_{K^{\p}\times I_{\p}}, \Di_{\chi_{\mathcal{S}}})_{\m_{\pi}}\too \HH^1(X_{K^{\p}\times I_{\p}}, \Co)_{\m_{\pi}}$$
is surjective.
\end{enumerate} 
\end{Thm}

The following proposition completes the picture by taking into account the cohomology in degree $2$.
\begin{Pro}
For every curve $\mathcal{S}$ as in Theorem \ref{thm:HBSW} \eqref{HBSW2} that is smooth at $\cf$ we have
$$\HH^2(X_{K^{\p}\times I_{\p}}, \Di_{\chi_{\mathcal{S}}})_{\m_{\pi}}\neq 0.$$
More precisely, the space is free of rank one over $\mathcal{O}_{\mathcal{S}, \cf}$ and the canonical map
$$\HH^2(X_{K^{\p}\times I_{\p}}, \Di_{\chi_{\mathcal{S}}})_{\m_{\pi}}\otimes_{\mathcal{O}_\mathcal{S}(\mathcal{S})}\Co\too \HH^2(X_{K^{\p}\times I_{\p}}, \Co)_{\m_{\pi}}$$
is an isomorphism.
Moreover, $\HH^2(X_{K^{\p}\times I_{\p}}, \Di_{\chi_{\mathcal{S}}})_{\m_{\pi}}$ and $\HH^1(X_{K^{\p}\times I_{\p}}, \Di_{\chi_{\mathcal{S}}})_{\m_{\pi}}$ are isomorphic as modules over the Hecke algebra.
\end{Pro}
\begin{proof}
Let $m$ be a generator of the maximal ideal of the localization.
From the short exact sequence
$$0 \rightarrow \Di_{\chi_{\mathcal{S}}} \stackrel{\times m}{\too} \Di_{\chi_{\mathcal{S}}} \rightarrow \Di_{\cf} \too 0$$
we obtain a long exact sequence in cohomology 
\begin{align*}
& \xrightarrow{\times m}\HH^{1}(X_{K^{\p}\times I_{\p}}, \Di_{\chi_{\mathcal{S}}})_{\m_{\pi}} \too \HH^1(X_{K^{\p}\times I_{\p}}, \Co)_{\m_{\pi}} \xrightarrow{0} \HH^2(X_{K^{\p}\times I_{\p}}, \Di_{\chi_{\mathcal{S}}})_{\m_{\pi}}  \\
& \xrightarrow{\times m} \HH^2(X_{K^{\p}\times I_{\p}}, \Di_{\chi_{\mathcal{S}}})_{\m_{\pi}} \too \HH^2(X_{K^{\p}\times I_{\p}}, \Co)_{\m_{\pi}} \too 0.
\end{align*}
Here the second map is the zero map as the first map is surjective by the second part of Theorem \ref{thm:HBSW}, and the last term is $0$ because the system of eigenvalues for $\pi$ doesn't appear in degrees greater by Theorem \ref{thm:HBSW} \eqref{HBSW1}. 

We then get that multiplication by $m$ is injective on $\HH^2(X_{K^{\p}\times I_{\p}}, \Di_{\chi_{\mathcal{S}}})_{\m_{\pi}}$ and, therefore, the map
$$\HH^2(X_{K^{\p}\times I_{\p}}, \Di_{\chi_{\mathcal{S}}})_{\m_{\pi}}\otimes_{\mathcal{O}_{\mathcal{S}, \cf}}\Co\too \HH^2(X_{K^{\p}\times I_{\p}}, \Co)_{\m_{\pi}}$$
is an isomorphism.
By hypothesis, $\HH^2(X_{K^{\p}\times I_{\p}}, \Co)_{\m_{\pi}}$ is one-dimensional.
Therefore, the $\mathcal{O}_{\mathcal{S}, \cf}$-module $\HH^2(X_{K^{\p}\times I_{\p}}, \Di_{\chi_{\mathcal{S}}})_{\m_{\pi}}$ is cyclic by Nakayama's Lemma.
If it were torsion, then multiplication by $m$ would not be injective on it, which is a contradiction; so it is free.

Let $r$ be a generator of the ideal of $\mathcal{O}_{\mathcal{U},\cf}$ corresponding to $\mathcal{S}$.
The modules $\HH^1(X_{K^{\p}\times I_{\p}}, \Di_{\chi_{\mathcal{S}}})_{\m_{\pi}}$ respectively $\HH^2(X_{K^{\p}\times I_{\p}}, \Di_{\chi_{\mathcal{S}}})_{\m_{\pi}}
$ are kernel respectively cokernel of the multiplication by $r$ map on $\HH^2(X_{K^{\p}\times I_{\p}}, \Di_{\chi_{\mathcal{U}}})_{\m_\pi}$.
Multiplication by the appropriate power of $r$ yields the sought-after isomorphism.
\end{proof}

We would like to apply Theorem \ref{thm:Linvderivative} to conclude that $\LI$-invariants are independent of the cohomological degree.
There are however several problems:
first, it is not clear whether one can find a curve $\mathcal{S}$ as in the theorem that is smooth at $\cf$.
Second, if one finds such a curve, the automorphic Colmez--Greenberg--Stevens formula in the inert case only produces one element in the intersection of $\LI^{(0)}(f,p)$ and $\LI^{(1)}(f,p)$.
As these spaces are two-dimensional by Proposition \ref{codimension} this gives us no new information.
In the split case the situation is better: one can at least compute the $\LI$-invariant of one of the primes lying above $p$ and, if the tangent space of $\mathcal{S}$ at $\cf$ is generic enough, one can compute both.
\begin{Cor}
Let $F$ be an imaginary quadratic field unramified at $p$ and let $\pi$ be a cuspidal Bianchi newform of parallel weight $2$.
Suppose that $p$ splits in $F$ as $p=\p \overline{\p}$ and $f$ is Steinberg at both, $\p$ and $\overline{\p}$. Then at least one of the equalities
\begin{align*}
\LI^{(0)}(\pi,\p)&=\LI^{(1)}(\pi,\p)\\
\intertext{or}
\LI^{(0)}(\pi,\overline{\p})&=\LI^{(1)}(\pi,\overline{\p})
\end{align*}
holds.
\end{Cor}

\begin{Rem}
It was shown in \cite{Ge3} that the theorem above would also follow from Venkatesh's conjectures on the action of derived Hecke algebras. 
\end{Rem}

The situation simplifies substantially if $\pi$ is the base change of a modular form $f$:
\begin{Thm}\label{thm:Bianchi}
Let $F$ be an imaginary quadratic field unramified at $p$.
Let $f$ be a newform of weight $2$ that is Steinberg at $p$ and $\BC(f)$ its base change to $F$.
The following holds:
\begin{enumerate}[(i)]
\item If $p=\p \overline{\p}$ is split in $F$, then
$$\LI^{(0)}(\BC(f),\p)=\LI^{(1)}(\BC(f),\p)=\LI^{(0)}(\BC(f),\overline{\p})=\LI^{(1)}(\BC(f),\overline{\p})=\LI(f,p).$$
\item If $p$ is inert, we have
$$\LI^{(0)}(\BC(f),p)\cap \Hom^{\cont}(\Q_p^{\ast},\Co)=\LI^{(1)}(\BC(f),p) \cap \Hom^{\cont}(\Q_p^{\ast},\Co)= \LI(f,p).$$
\item If $p$ is inert, we have
$$\LI^{(0)}(\BC(f),p)=\LI^{(1)}(\BC(f),p).$$
\end{enumerate} 
\end{Thm}
\begin{proof}
In Section 5 of \cite{BWi} it is shown that on may take $\mathcal{S}\subseteq \Ws$ to be the parallel weight curve (see Section 5 of \cite{BWi}), which is clearly smooth at $\cf$.
Using $p$-adic Langlands functoriality as explained in \textit{loc.~cit.~} the first two claims follow by applying Theorem \ref{thm:Linvderivative} (respectively its analogue in the Bianchi setting) for the modular form $f$ and its base change to $F$.
The last claim is a consequence of the second claim and the Galois invariance of automorphic $\LI$-invariants attached to base change representations (see Lemma 3.1 of \cite{Ge2}).
\end{proof}

\begin{Rem}
In the inert case the equality $\LI^{(0)}(\BC(f),p)\cap \Hom^{\cont}(\Q_p^{\ast},\Co)=\LI(f,p)$ was proven in \cite{Ge2}, Lemma 3.3, using Artin formalism for $p$-adic $L$-functions and the exceptional zero formula.
An analogous result for higher weight forms was proven by Barrera-Salazar and Williams in \cite{BWi}, Proposition 10.2.

Our approach can be adapted to higher weights making the construction of Stark--Heegner cycles of \cite{VW} unconditional in the base change case.
\end{Rem}